\documentclass[11pt]{amsart}

\usepackage{pdfsync}
\usepackage[ansinew]{inputenc}
\usepackage{amsmath}
\usepackage{amssymb}
\usepackage{pigpen}
\usepackage[curve]{xypic}
\usepackage{amsthm}
\usepackage{hyperref}
\usepackage{graphicx}
\usepackage{tikz}
\usetikzlibrary{shapes,arrows,calc}
\usepackage{enumitem}
\newdir{ >}{{}*!/-7pt/\dir{>}}

\newtheorem{satz}{Satz}[section]
\newtheorem{Theorem}[satz]{Theorem}
\newtheorem*{MainTheorem}{Main Theorem}
\newtheorem{Definition}[satz]{Definition}
\newtheorem{Lemma}[satz]{Lemma}

\newtheorem{Prop}[satz]{Proposition}

\newtheorem*{Conj*}{Conjecture}
\newtheorem{Question}{Question}

\theoremstyle{remark}
\newtheorem{example}[satz]{Example}
\newtheorem{examples}[satz]{Examples}

\newtheorem{remark}[satz]{Remark}

\newcommand{\F}{\mathcal{F}}

\newcommand{\Z}{\mathbb{Z}}

\newcommand{\C}{\mathcal{C}}
\newcommand{\K}{\mathcal{K}}
\newcommand{\D}{\mathcal{D}}

\newcommand{\un}{\underline{n}}
\newcommand{\setms}{\!\setminus\!}

\newcommand{\MM}{\mathcal{M}}
\newcommand{\CC}{\mathcal{C}}
\newcommand{\KK}{\mathcal{K}}
\newcommand{\FF}{\mathcal{F}}
\newcommand{\WW}{\mathcal{W}}
\newcommand{\DD}{\mathcal{D}}

\DeclareMathOperator{\colim}{colim}

\DeclareMathOperator{\Sd}{Sd}
\DeclareMathOperator{\id}{id}

\DeclareMathOperator{\Cat}{Cat}
\DeclareMathOperator{\sSet}{sSet}
\DeclareMathOperator{\Top}{Top}

\DeclareMathOperator{\Fun}{Fun}
\DeclareMathOperator{\Ex}{Ex}

\newcommand{\wt}{\widetilde}

\newcommand{\unn}{\un \setms n}
\newcommand{\del}{\partial}

\newenvironment{pf}{\begin{proof}[Proof]}{\end{proof}}

\numberwithin{equation}{section}

\makeatletter
\def\namedlabel#1#2{\begingroup
    #2%
    \def\@currentlabel{#2}%
    \phantomsection\label{#1}\endgroup
}
\makeatother

\author{Lennart Meier}
\address{ University of Virginia,  Charlottesville VA 22904, USA}
\email{flm5z@virginia.edu} 

\author{Viktoriya Ozornova}
\address{Fachbereich Mathematik, Universit\"{a}t Bremen, 28359 Bremen, Germany}
\email{ozornova@math.uni-bremen.de} 

\title{Fibrancy of Partial Model Categories}

\begin{document}
\maketitle
\begin{abstract}
We investigate fibrancy conditions in the Thomason model structure on the category of small categories. In particular, we show that the category of weak equivalences of a partial model category is fibrant. Furthermore, we describe connections to calculi of fractions. 
\end{abstract}

\section{Introduction}
The homotopical study of (small) categories with respect to the nerve functor
\[N\colon \Cat\to \sSet \]
was started by Segal \cite{Segal} and Quillen \cite{Quillen}. But it was only Thomason in his 1980 paper \cite{ThomasonCat} who equipped Cat with a model structure, where a functor $f\colon \C \to \D$ is a weak equivalence if and only if $Nf\colon N\C \to N\D$ is a weak equivalence of simplicial set. The choice of fibrations is less obvious: If $\Ex\colon \sSet \to \sSet$ denotes Kan's Ex-functor, a functor $f\colon \C\to \D$ is a fibration if and only if $\Ex^2Nf: \Ex^2N\C \to \Ex^2N\D$ is a Kan fibration. Taking just functors $f$ such that $Nf$ or $\Ex Nf$ are fibrations would not define enough fibrations in $\Cat$. 

It is classically known that a category $\C$ is a groupoid if and only if $N\C$ is a fibrant simplicial set. We will give a characterization by Fritsch and Latch of categories $\C$ such that $\Ex N\C$ is fibrant, which appeared previously only without proof in print (as far as the authors know). 

A convenient characterization of categories $\C$ such that $\Ex^2 N\C$ is fibrant (i.e., $\C$ is fibrant in the Thomason model structure) is unknown to this date. This problem was already considered, for example, by Beke \cite{BekeFibrations}. In this note, we will present a sufficient criterion in form of a large class of fibrant categories motivated by homotopy theory: A relative category is a category $\MM$ together with a subcategory $\WW$ of $\MM$ containing all objects. Here, $\WW$ is thought of as the ``weak equivalences'' in $\MM$. In general, relative categories are not well-behaved. For example, the functor $\MM \to \MM[\WW^{-1}]$ from a relative category to its homotopy category (i.e., its localization at $\WW$) may send morphisms to isomorphisms that are not weak equivalences. In \cite{BarwickKanNeu}, Barwick and Kan introduced the notion of a partial model category, which is a convenient notion of a relative category with a ``3-arrow calculus" in the sense of Dwyer, Hirschhorn, Kan and Smith 
\cite{DHKS}. Among 
several pleasant properties of partial 
model categories, we want to mention that for a partial model category $(\MM,\WW)$ the 
functor $\MM\to \MM[\WW^{-1}]$ sends only the weak equivalences to isomorphisms.

Our main result is the following. 

\begin{MainTheorem}
Let $(\MM,\WW)$ be a partial model category. Then $\WW$ is fibrant in the Thomason model structure.
\end{MainTheorem}

This includes, in particular, all categories that possess all pushouts or all pullbacks, but is considerably more general. For example, the category of weak equivalences of a model category (such as $\Top$, $\sSet$, ...) usually does not possess all pushouts or all pullbacks, but is fibrant in the Thomason model structure by the theorem above.

The structure of this article is as follows: In Section \ref{Left calculus of fractions and fibrancy}, we will give the criterion by Fritsch and Latch for the fibrancy of $\Ex N\C$. In Section \ref{Partial model categories}, we will define partial model categories and discuss some examples. In Section \ref{Fibrancy of Partial Model Categories}, we will prove our main theorem. In Section \ref{Examples}, we will discuss some concrete examples and also give criteria when a category has a left calculus of fractions. We end with Section \ref{OpenQuestions}, containing open questions and further remarks. 

\subsection*{Acknowledgements}
We thank the University of Bremen and the University of Virginia for their hospitality during our visits. 

\subsection*{Conventions and notation}
We denote by $\Cat$ the category of small categories, by $\sSet$ the category of simplicial sets and by $\Top$ the category of topological spaces.  

There is an adjunction
\[\xymatrix{\Cat \ar@/^/[r]^N & \sSet \ar@/^/[l]^c }\]
where $N$ denotes the \textit{nerve functor} given by $(N\CC)_n = \Cat([n], \CC)$ and its left adjoint $c$ is often called the \textit{fundamental category functor}. 

Another important adjunction for our purposes is 
\[\xymatrix{\sSet \ar@/^/[r]^{\Ex} & \sSet \ar@/^/[l]^{\Sd} }\]
where $\Sd$ is the Kan subdivision functor and $\Ex$ is its right adjoint. More concretely, $\Sd \Delta[n]$ is the nerve of the poset that has as objects all non-degenerate simplices of $\Delta[n]$ and the relation is generated by $v \leq w$ if $v$ is a face of $w$. For a general $X$, we have $\Sd X = \colim_{\Delta[n]\to X} \Sd \Delta[n]$. Kan's $\Ex$ is defined by $(\Ex X)_n = \sSet(\Sd \Delta[n], X)$. 

We will always equip $\sSet$ with the Quillen model structure, where the fibrant objects are exactly the Kan complexes.   \\

For set-theoretic reasons, we have to define $\Cat$ to be the category of \textit{small} categories, i.e., categories with a \textit{set} of objects. As, for example, the category $\Top$ of topological spaces in the usual sense is not small, there appears to be a technical problem, if we want to state that the category of weak equivalences of $\Top$ is fibrant in $\Cat$. There are at least two possible remedies for this.

The first is to view the statement that $\Ex^2 N\CC$ is Kan for a category $\CC$ just as a formal statement. It is equivalent to $\CC$ having the right lifting property with respect to all inclusions $c\Sd^2 \Lambda^i[n] \to c\Sd^2 \Delta[n]$. This statement makes sense also for large categories $\CC$ and it is actually this lifting property that we will prove. 

The second possibility is to assume the existence of a Grothendieck universe $U$ and call its elements \textit{small sets}. We would then redefine $\Top$ to consist just of all topological spaces whose underlying set is small. Then $\Top$ is a small category. Note that we have to be careful in this case what we mean by a model category if $\Top$ should still be a model category. We can only assume the existence of limits and colimits over categories with a small set of objects. 

The reader might choose the remedy he or she likes and we will ignore this issue for the rest of this article. 

\section{Left calculus of fractions and fibrancy}
\label{Left calculus of fractions and fibrancy}
The goal of this section is to prove the following theorem.
\begin{Theorem}\label{CalculusTheorem}\cite{LatchThomasonWilson} Let $\CC$ be a category. Then $\Ex N\CC$ is Kan if and only if $\CC$ admits a left calculus of fractions.\end{Theorem}

The content of this section is not new. The theorem above seems to appear first in \cite{LatchThomasonWilson}, and is also mentioned in \cite{BekeFibrations}. The second author also had a helpful e-mail exchange with Tibor Beke on this topic. We give here a proof since we were unable to find one in the literature. First we will recall the notion of a left calculus of fractions. 

\begin{Definition}\label{CalcFrac}\cite{GabrielZisman} 
 A small category $\mathcal{C}$ is said to have \textbf{left calculus of fractions} (with respect to itself) if it satisfies the following to conditions:
 \begin{enumerate}[label=\textbf{(CF\arabic*)}, ref=(CF\arabic*)]
  \item \label{CF1} For any two morphisms $s\colon X\to Y$, $t\colon X \to Z$ in $\mathcal{C}$, there is a commutative diagram of the form
  \begin{eqnarray*}
   \xymatrix{
   X\ar[r]^{s}\ar[d]^{t} & Y \ar@{-->}[d]\\
   Z \ar@{-->}[r] & W.
   }
  \end{eqnarray*}
\item \label{CF2} For any three morphisms $f,g \colon X\to Y$ and $s \colon X' \to X$ satisfying $fs=gs$, there is a further morphism $t\colon Y \to Y'$ so that $ tf=tg$. 
 \[\xymatrix{
X' \ar[r]^{s} & X \ar@/^/[r]^{f} \ar@/_/[r]_{g} & Y \ar@{-->}[r] & Y'
}\]
 \end{enumerate}

\end{Definition}

In a monoid language, the condition \ref{CF1} corresponds to the existence of a (right) common multiple for a pair of morphisms. The condition \ref{CF2} is a weak version of cancellativity (see Example \ref{MonoidExample} for a non-cancellative monoid having a left calculus of fractions). The following lemma says that we also obtain common multiples for arbitrary finite sets of morphisms with a common source. 

\begin{Lemma}\label{GenCM}
 Let $f_i\colon X \to X_i$ for $1\leq i \leq n$, $n\geq 2$, be a finite set of morphisms with the same source in a category $\C$ possessing left calculus of fractions. Then there are morphisms $g_i\colon X_i\to Y$ so that $g_if_i=g_jf_j$ for all $1\leq i,j\leq n$. 
\end{Lemma}

\begin{pf}
 This follows by easy induction. The base case is exactly Condition \ref{CF1}. Assuming the claim holds for some $n-1$, let  $f_i\colon X \to X_i$ for $1\leq i \leq n$, $n\geq 3$, be given. Then we already have  morphisms $g_i \colon X_i\to Y$ for $1\leq i \leq n-1$ with $g_if_i=g_jf_j$ by induction hypothesis. Now we apply Condition \ref{CF1} to the morphisms $g_1f_1 \colon X \to Y$, $f_n \colon X \to X_n$, so we obtain maps $h\colon Y \to Z$ and $k\colon X_n \to Z$ with $hg_1f_1=kf_n$. Now the set of functions $hg_1, \ldots, hg_{n-1}, k$ does the job. 
\end{pf}

\begin{Prop}\cite{LatchThomasonWilson} 
If a category $\mathcal{C}$ has left calculus of fractions, then $\Ex N\C$ is Kan.
\end{Prop}
\begin{proof}
Assume that we are given a map $\Lambda^k[n] \to \Ex N\C$ that we want to extend to a map $\Delta[n]\to \Ex N\C$. Consider the adjoint map (functor) $F\colon c\Sd\Lambda^k[n] \to \C$. As explained in Section 3 of \cite{FiorePaoli}, the category $c\Sd \Delta[n]$ is the poset of non-empty subsets $\mathcal{P}_+(\underline{n})$ of $\underline{n}=\{0,1,\ldots, n\}$ with inclusions as morphisms. Furthermore,   
\begin{eqnarray*}
c\Sd\Lambda^k[n] =\mathcal{P}_+(\underline{n}) \setminus \{\underline{n}, \underline{n}\setminus k\},
\end{eqnarray*}
considered as a subposet of $c\Sd \Delta[n]$. Here and in the following, we abbreviate $\underline{n}\setminus \{i\}$ to $\underline{n}\setminus i$.

To extend $F$ to a functor $c\Sd \Delta[n] \to \CC$, corresponds in this notation to the following: We have to assign to all inclusions $ \underline{n}\setminus \{i,k\}\to  \underline{n}\setminus \{k\}$ for $i\neq k$ and to all inclusions $ \underline{n}\setminus \{j\} \to  \underline{n}$ morphisms in $\CC$ and check that this defines a functor. As a first step, we define a functor $\wt{F} \colon \mathcal{P}_+(\underline{n}) \setminus \{\underline{n}\setminus k\} \to \C$ extending $F$. Set
\begin{eqnarray*}
a_i:=F(\{k\} \to \underline{n}\setminus i)
\end{eqnarray*}
for $i\neq k$. 
By Lemma \ref{GenCM}, there are morphisms $b_i\colon F(\underline{n}\setminus i) \to Z$ in $\C$ so that $b_1a_1=\ldots =b_na_n$ (omitting index $k$). 

We tentatively extend $F$ by $\wt{F}_{\mathrm{tent}}(\underline{n})=Z$ and $\wt{F}_{\mathrm{tent}}(\underline{n}\setminus i \to \underline{n})=b_i$.  This does not in general define a functor, as we will see below, but we will find a way to ``correct'' this definition to a functor. The only (possible) problem of functoriality arises for maps with target $\underline{n}$. It will be enough to consider for all pairs of distinct $i,j$ (unequal to $k$) the following square (and its image unter $\wt{F}_{\mathrm{tent}}$). 
\begin{eqnarray*}
\xymatrix{
 \underline{n}\setminus \{i,j\}\ar[r]\ar[d] &\underline{n}\setminus i\ar[d]\\
\underline{n}\setminus j \ar[r]& \underline{n}.
}
\end{eqnarray*}
Write $s=F(\{k\}\to \underline{n}\setminus \{i,j\})$. Since $F$ is a functor, we have 
\begin{eqnarray*}
F( \underline{n}\setminus \{i,j\}\to  \underline{n}\setminus j)s=a_j\\
F( \underline{n}\setminus \{i,j\}\to  \underline{n}\setminus i)s=a_i.
\end{eqnarray*}
Postcomposing by $b_j$ and $b_i$, respectively, we see that the image of the square will commute when precomposed with $s$. Yet, it does not necessarily commute as Condition \ref{CF2} is weaker than cancellativity. But we can still apply Condition \ref{CF2} to obtain a map $t_{ij} \colon Z \to W_{ij}$ so that the image of the square above commutes when postcomposed with $t_{ij}$. By Lemma \ref{GenCM}, we can find morphisms $u_{ij} \colon W_{ij} \to W$ in $\C$ for all  pairs of distinct $i,j$ (unequal to $k$) so that $v=u_{ij}t_{ij} \colon Z \to W$ is the same morphism for all $i,j$. 

Now define $\wt{F}(\underline{n})=W$ and $\wt{F}(\underline{n}\setminus i \to \underline{n})=vb_i$. We have to check that this extension of $F$ is now a functor. Again, the only check needed is for maps with target  $\underline{n}$. Assume we have two different chains of inclusions $A \subset A_1 \subset A_2 \subset \ldots A_r \subset \underline{n}$ and $A \subset B_1 \subset B_2 \subset \ldots B_s \subset \underline{n}$, where each inclusion may be assumed to enlarge the foregoing set by one element. Then either $B_s=A_r$, and we are done since $F$ was assumed to be a functor, or $A_r=\underline{n}\setminus i$ and $B_s=\underline{n}\setminus j$ for some $i\neq j$. Then $A \subset \underline{n}\setminus \{i,j\}$, and both morphisms factor through $\underline{n}\setminus \{i,j\}$, so that it is enough to show for all pairs of distinct $i,j$ (unequal to $k$) that the following square is mapped by $\wt{F}$ to a commutative square:
\begin{eqnarray*}
\xymatrix{
 \underline{n}\setminus \{i,j\}\ar[r]\ar[d] &\underline{n}\setminus i\ar[d]\\
\underline{n}\setminus j \ar[r]& \underline{n}.
}
\end{eqnarray*}
This is exactly achieved by the construction above. So $\wt{F}$ is indeed a functor extending $F$ to $\mathcal{P}_+(\underline{n}) \setminus \{\underline{n}\setminus \{k\}\}$.

Next, we have to extend the functor to all of $\mathcal{P}_+(\underline{n})$. Set $\wt{F}(\underline{n}\setminus k)=W$ and $\wt{F}(\underline{n}\setminus k \to \underline{n})=\id_{W}$. We are then forced to define
\begin{eqnarray*}
\wt{F}(\underline{n}\setminus \{i,k\} \to \underline{n}\setminus k)=\wt{F}(\underline{n}\setminus \{i,k\} \to \underline{n})
\end{eqnarray*}
for $i\neq k$. We have to check that this defines a functor $\mathcal{P}_+(\underline{n}) \to \C$. Note that it is immediate from the definition that compositions ending with $\underline{n}$ are mapped to the same morphisms in $\C$. By the same argument as before, it is enough to consider squares of the form 
\begin{eqnarray*}
\xymatrix{
 \underline{n}\setminus \{i,j,k\}\ar[r]\ar[d] &\underline{n}\setminus\{i,k\}\ar[d]\\
\underline{n}\setminus \{j,k\} \ar[r]& \underline{n}\setminus k.
}
\end{eqnarray*}
Since $\wt{F}(\underline{n}\setminus k \to \underline{n})$ is identity of $W$ and in particular a monomorphism, the claim follows from the commutativity of the outer square and the triangles in the following diagram: 
\begin{eqnarray*}
\xymatrix{
 \underline{n}\setminus \{i,j,k\}\ar[r]\ar[d] &\underline{n}\setminus\{i,k\}\ar[d]\ar@/^/[ddr]&\\
\underline{n}\setminus \{j,k\} \ar[r]  \ar@/_/[drr] & \underline{n}\setminus k\ar[rd]& \\
&&\underline{n}.
}
\end{eqnarray*}
This completes the proof of the fact that for any category $\C$ with left calculus of fractions, the simplicial set $\Ex N\C$ is Kan.
\end{proof}

\begin{Lemma}\cite{LatchThomasonWilson} 
Let $\C$ be a small category such that $\Ex N\C$ is Kan. Then any two morphism of $\C$ admit a common multiple, i.e., the Condition \ref{CF1} is satisfied. 
\end{Lemma}

\begin{proof}
Let $\C$ be a small category such that $\Ex N\C$ is a Kan simplicial set. Consider any two morphisms $s\colon X\to Y$, $t\colon X \to Z$ in $\mathcal{C}$. We define a map $F\colon c \Sd\Lambda^0[2] \to \C$ via 
\begin{center}
\begin{tikzpicture}
\begin{scope}[xshift=-3cm]
  \node (a) at (-2,2) {$01$};
  \node (b) at (0,2) {$02$};
  \node (d) at (-2,0) {$0$};
  \node (e) at (0,0) {$1$};
  \node (f) at (2,0) {$2$};
 \draw [->] (d) - - (a) node[midway,left] (K1){};
 \draw [<-] (b) -- (f)node[midway,left] (K3){};
  \draw [->] (d) -- (b)node[near end, above] (K2){};
 \draw[preaction={draw=white, -,line width=6pt}] (a) -- (e) node[near end,right] (K4){};
 \end{scope}
	\node (start) at (-1,1) {};
  \node (end) at (0.3,1) {};
\draw [|->] (start) -- (end);
 \begin{scope}[xshift=3cm]
  \node (a1) at (-2,2) {$Y$};
  \node (b1) at (0,2) {$Z$};
  \node (d1) at (-2,0) {$X$};
  \node (e1) at (0,0) {$Y$};
  \node (f1) at (2,0) {$Z$};
 \draw [->] (d1) - - (a1) node[midway,left] (L1){$s$};
 \draw [->] (f1) -- (b1)node[midway,left] (L3){$\id_Z$};
  \draw [->] (d1) -- (b1)node[near end, above] (L2){$t$};
 \draw[preaction={draw=white, -,line width=6pt}] [<-] (a1) -- (e1) node[near end,right] (L4){$\id_Y$};
 \end{scope}
\end{tikzpicture}
\end{center}

This diagram can be extended, due to the Kan property, to a functor $\wt{F}$ on

\begin{center}
\begin{tikzpicture}
  \node (max) at (0,4) {$012$};
  \node (a) at (-2,2) {$01$};
  \node (b) at (0,2) {$02$};
  \node (c) at (2,2) {$12$};
  \node (d) at (-2,0) {$0$};
  \node (e) at (0,0) {$1$};
  \node (f) at (2,0) {$2$};
  \draw [->] (d) -- (a);
	\draw [->] (d) -- (b);
	\draw [->] (f) -- (b);
	\draw [->] (f) -- (c);
	\draw[preaction={draw=white, -,line width=6pt}] [->] (e) -- (a);
	\draw[preaction={draw=white, -,line width=6pt}] [->] (e) -- (c);
	\draw [->] (a) -- (max);
	\draw [->] (b) -- (max);
	\draw [->] (c) -- (max);
\end{tikzpicture}
\end{center}
In particular, the image of the commutative square 
\begin{eqnarray*}
\xymatrix{
 0\ar[r]\ar[d] & 01\ar[d]\\
02 \ar[r]& 012
}
\end{eqnarray*}
in $c\Sd \Delta[2]$ under $\wt{F}$ yields the desired commutative square in $\C$. 
\end{proof}

\begin{Lemma}\cite{LatchThomasonWilson} 
 Let $\C$ be a small category so that $\Ex N\C$ is a Kan simplicial set. Then $\C$ satisfies Condition \ref{CF2}. 
\end{Lemma}

\begin{proof}
 
 Let $f,g \colon X\to Y$ and $s \colon X' \to X$ be morphisms in $\C$ so that $fs=gs$. We will prove the existence of a morphism $t \colon Y \to Y'$ with $tf=tg$ by filling a $\Lambda_0[3]$-horn in $\Ex N\C$. Using the adjunctions again, we give first a functor $F \colon c \Sd \Lambda_0[3] \to \C$. We will see later why its extension $\wt{F} \colon c \Sd \Delta[3] \to \C$ yields the desired morphism $t$. 

 Recall that $c \Sd \Lambda_0[3]$ has the following shape. 

 \begin{center}
\begin{tikzpicture}[scale=1.1]
  \node (a) at (-3,2) {$012$};
  \node (b) at (0,2) {$013$};
  \node (c) at (3,2) {$023$};

  \node (e) at (-3,0){$01$};
  \node (f) at (-1,0){$02$};
  \node (g) at (1,0){$03$};
  \node (h) at (3,0){$12$};
  \node (i) at (5,0){$13$};
  \node (j) at (7,0){$23$};
  \node (k) at (-3,-2) {$0$};
  \node (l) at (0,-2) {$1$};
  \node (m) at (3,-2) {$2$};
  \node (n) at (6,-2) {$3$};

\draw [->] (e)--(a);
\draw [->] (f)--(a);
\draw [->] (h)--(a);
\draw[blue] [->] (e)--(b);
\draw[blue] [->] (g)--(b);
\draw[blue] [->] (i)--(b);
\draw[red] [->] (f)--(c);
\draw[red] [->] (g)--(c);
\draw[red] [->] (j)--(c);

\draw[yellow] [<-] (e)--(k);
\draw[yellow] [<-] (f)--(k);
\draw[yellow] [<-] (g)--(k);
\draw[cyan] [<-] (e)--(l);
\draw[cyan] [<-] (h)--(l);
\draw[cyan] [<-] (i)--(l);
\draw[magenta] [<-] (f)--(m);
\draw[magenta] [<-] (h)--(m);
\draw[magenta] [<-] (j)--(m);
\draw[brown] [<-] (g)--(n);
\draw[brown] [<-] (i)--(n);
\draw[brown] [<-] (j)--(n);
\end{tikzpicture}
\end{center}

We now define a functor from this category to $\C$. 

\begin{center}
\begin{tikzpicture}[scale=1.2]
\def\grl{2.5}
\def\kll{\grl/2}
  \node (a) at (-3,2) {$Y$};
  \node (b) at (-3+\grl,2) {$Y$};
  \node (c) at (-3+2*\grl,2) {$Y$};

  \node (e) at (-3,0){$X$};
  \node (f) at (-3+\kll,0){$Y$};
  \node (g) at (-3+2*\kll,0){$X$};
  \node (h) at (-3+3*\kll,0){$X$};
  \node (i) at (-3+4*\kll,0){$X$};
  \node (j) at (-3+5*\kll,0){$X$};
  \node (k) at (-3,-2) {$X'$};
  \node (l) at (-3+\grl,-2) {$X'$};
  \node (m) at (-3+2*\grl,-2) {$X$};
  \node (n) at (-3+3*\grl,-2) {$X$};

\draw [->] (e)--(a) node[midway,left] (L1){$f$};
\draw [->] (f)--(a) node[pos=0.6, right] (L2){$\id_Y$};
\draw [->] (h)--(a) node[pos=0.5, below left] (L3){$g$};
\draw[blue] [->] (e)--(b) node[pos=0.4, right] (L4){$g$};
\draw[blue] [->] (g)--(b)  node[pos=0.6,left] (L5){$g$};
\draw[blue] [->] (i)--(b) node[midway,left] (L6){$g$};
\draw[red] [->] (f)--(c) node[pos=0.6,left] (L7){$\id_Y$};
\draw[red] [->] (g)--(c)  node[pos=0.3,right] (L8){$g$};
\draw[red] [->] (j)--(c) node[midway,left] (L9){$g$};

\draw[green] [<-] (e)--(k) node[midway,left] (M1){$s$};
\draw[green] [<-] (f)--(k) node[midway,left] (M2){$fs$};
\draw[green] [<-] (g)--(k) node[pos=0.75,right] (M3){$s$};
\draw[cyan] [<-] (e)--(l)node[pos=0.75,left] (M4){$s$};
\draw[cyan] [<-] (h)--(l)node[pos=0.75,left] (M5){$s$};
\draw[cyan] [<-] (i)--(l)node[pos=0.8,right] (M6){$s$};
\draw[magenta] [<-] (f)--(m)node[pos=0.4, below left] (M7){$g$};
\draw[magenta] [<-] (h)--(m) node[pos=0.6,left] (M8){$\id_X$};
\draw[magenta] [<-] (j)--(m)node[pos=0.8,right] (M9){$\id_X$};
\draw[brown] [<-] (g)--(n)node[pos=0.4,right] (M10){$\id_X$};
\draw[brown] [<-] (i)--(n)node[pos=0.4,right] (M11){$\id_X$};
\draw[brown] [<-] (j)--(n)node[midway,right] (M12){$\id_X$};
\end{tikzpicture}
\end{center}
 
 One easily checks that this, indeed, defines a functor (there are $9$ squares in this diagram, which can be seen to commute). By the Kan extension property and using the adjunctions again, we obtain a functor from $c\Sd \Delta[3]$ to $\C$, as displayed in the following picture. 
 \begin{center}
\begin{tikzpicture}[scale=1.2]
 \def\grl{2.5}
\def\kll{1.75}
  \node (a) at (-3,2) {$Y$};
  \node (b) at (-3+\grl,2) {$Y$};
  \node (c) at (-3+2*\grl,2) {$Y$};

  \node (e) at (-3,0){$X$};
  \node (f) at (-3+\kll,0){$Y$};
  \node (g) at (-3+2*\kll,0){$X$};
  \node (h) at (-3+3*\kll,0){$X$};
  \node (i) at (-3+4*\kll,0){$X$};
  \node (j) at (-3+5*\kll,0){$X$};
  \node (k) at (-3,-2) {$X'$};
  \node (l) at (-3+\grl,-2) {$X'$};
  \node (m) at (-3+2*\grl,-2) {$X$};
  \node (n) at (-3+3*\grl,-2) {$X$};

\draw [->] (e)--(a) node[midway,left] (L1){$f$};
\draw [->] (f)--(a) node[pos=0.6, right] (L2){$\id_Y$};
\draw [->] (h)--(a) node[pos=0.5, below left] (L3){$g$};
\draw[blue] [->] (e)--(b) node[pos=0.4, right] (L4){$g$};
\draw[blue] [->] (g)--(b)  node[pos=0.6,left] (L5){$g$};
\draw[blue] [->] (i)--(b) node[pos=0.45, below left] (L6){$g$};
\draw[red] [->] (f)--(c) node[pos=0.6,left] (L7){$\id_Y$};
\draw[red] [->] (g)--(c)  node[pos=0.3,right] (L8){$g$};
\draw[red] [->] (j)--(c) node[midway, below left] (L9){$g$};

\draw[green] [<-](e)--(k) node[midway,left] (M1){$s$};
\draw[green] [<-] (f)--(k) node[midway,left] (M2){$fs$};
\draw[green] [<-] (g)--(k) node[pos=0.75,right] (M3){$s$};
\draw[cyan] [<-] (e)--(l)node[pos=0.75,left] (M4){$s$};
\draw[cyan] [<-] (h)--(l)node[pos=0.75,left] (M5){$s$};
\draw[cyan] [<-] (i)--(l)node[pos=0.8,right] (M6){$s$};
\draw[magenta] [<-] (f)--(m)node[pos=0.4,left] (M7){$g$};
\draw[magenta] [<-] (h)--(m) node[pos=0.6,left] (M8){$\id_X$};
\draw[magenta] [<-] (j)--(m)node[pos=0.8,right] (M9){$\id_X$};
\draw[brown] [<-] (g)--(n)node[pos=0.55,right] (M10){$\id_X$};
\draw[brown] [<-] (i)--(n)node[pos=0.4,right] (M11){$\id_X$};
\draw[brown] [<-] (j)--(n)node[midway,right] (M12){$\id_X$};

 \node (max) at (-3+1.5*\grl,4) {$W$};
 \node (d) at (-3+3*\grl,2) {$Z$};
 \draw[thick] [->] (a)--(max) node[midway, left](N1){$b_3$};
\draw[thick] [->] (b)--(max)node[midway, left](N2){$b_2$};
\draw [thick] [->](c)--(max)node[midway, left](N3){$b_1$};
\draw[thick] [->] (d)--(max)node[midway, left](N4){$b_0$};

\draw[thick] [->] (h)--(d)node[midway, left](O1){$a_3$};
\draw[thick] [->] (i)--(d)node[pos=0.15, right](O2){$a_2$};
\draw[thick] [->] (j)--(d)node[midway, left](O3){$a_1$}; 
 
\end{tikzpicture}
\end{center}
Now this implies $b_2g=b_1g$, $b_1=b_3$ and $b_2g=b_3f$, so $b_1f=b_1g$ and $t=b_1$ does the job. This completes the proof. 
\end{proof}
This finishes the proof of Theorem \ref{CalculusTheorem}

\section{Partial model categories}
\label{Partial model categories}

\begin{Definition}A \textit{relative category} is a category $\MM$ together with a subcategory $\WW$ containing every object. The morphisms in $\WW$ are often called \textit{weak equivalences}.\end{Definition} 

A general relative category is difficult to work with. For example, one can form the localization $\MM \to \MM[\WW^{-1}]$, but in general not every preimage of an isomorphism will be a weak equivalence. One useful class of relative categories is the following:

\begin{Definition}\cite{BarwickKanNeu} A relative category $(\MM,\WW)$ is called a \textbf{partial model category} if there are subclasses $\CC, \FF\subset \WW$ (called \textit{(acyclic) cofibrations} and \textit{(acyclic) fibrations}, respectively) satisfying the following axioms:
\begin{enumerate}
\item $\WW$ satisfies the $2$ out of $6$ property, i.e., if $r,s$ and $t$ are morphisms such that the compositions $sr$ and $ts$ exist and are in $\WW$, then $r,s,t$ and $tsr$ are also in $\WW$.
\item For every map $f\in\CC$, its pushouts along arbitrary maps in $\MM$ exist and are again in $\CC$.
\item For every map $f\in\FF$, its pullbacks along arbitrary maps in $\MM$ exist and are again in $\FF$.
\item Every weak equivalence can be functorially factorized into a cofibration and a fibration, i.e., there is a functor
\begin{eqnarray*}
 F=(F_c, F_f)\colon \Fun(0\to 1, \mathcal{W}) \to \Fun(0\to 1, \C) \times \Fun(0\to 1, \F),
\end{eqnarray*}
such that for every morphism $g$ in $\mathcal{W}$, the morphisms $F_c(g), F_f(g)$ are composable and $F_f(g)\circ F_c(g)=g$ holds. 
\end{enumerate}
\end{Definition}

This is a slightly more restrictive variant of the notion of a homotopical category with $3$-arrow calculus as in \cite{DHKS}. 

As in \cite{BarwickKanNeu}, we have the following list of examples:
\begin{examples}
\begin{enumerate}
\item For every model category its underlying relative category is a
partial model category.
\item Let $(\MM,\WW)$ be a partial model category and $\MM_1\subset \MM$ be a full subcategory with the property that if $X\in\mathrm{Ob}\MM_1$ and $Y\in\mathrm{Ob}\MM$ are connected by a zig-zag of weak equivalences, then $Y\in \mathrm{Ob}\MM_1$. Such subcategories are called \textbf{homotopically full}. Then $(\MM_1, \WW\cap \MM_1)$ is a partial model category.
\item For every partial model category $(\MM, \WW)$ and category $\DD$,
the functor relative category $(\MM, \WW)^\DD = (\MM^\DD, \WW^\DD)$ is again a partial model
category.
\item If $\MM$ is a category with all pullbacks, then $(\MM, \MM)$ is a partial model category with $\CC$ only consisting of identities and $\FF = \MM$. 
\item If $\MM$ is a category with all pushouts, then $(\MM, \MM)$ is a partial model category with $\FF$ only consisting of identities and $\CC = \MM$.
\item For every partial model category $(\MM, \WW)$ the associated relative category $(\WW, \WW)$ is a
partial model category.
\end{enumerate}
\end{examples}

\section{Fibrancy of Partial Model Categories}
\label{Fibrancy of Partial Model Categories}
Let now $\WW$ be the category of weak equivalences of a partial model category or, equivalently, a partial model category, where every morphism is a weak equivalence. Denote the distinguished classes of cofibrations by $\C$ and that of fibrations by $\F$. Our aim in this section is to show that the simplicial set $\Ex^2N\WW$ is Kan. 

For a category $\D$, let $\K(\D)$ be the category $\D \times (0\to 1) \cup_{\D \times 1} \D^{\vartriangleleft}$, where $\D^{\vartriangleleft}$ denotes the category $\D$ with an additional initial object. Thus, $\K(\D)$ consists of two copies of $\D$, where there is a unique map from the $0$-copy of each object to the $1$-copy of it, and each object in the $1$-copy receives an additional morphism from a ``partial'' initial object. We will often consider the inclusion of $\D \cong \D \times 0$ into $\K(\D)$. We will denote the ``partial'' initial object by $k_\DD \in \K(\DD)$ or, if no confusion is possible, just by $k$. 

\begin{remark}\label{PropertiesK}
\begin{enumerate}
 \item  The assignment $\K \colon \Cat \to \Cat$ constitutes a functor. For a functor $F \colon \D \to \D'$, define $\K(F)\colon \K(\D)\to \K(\D')$ to be a copy of $F$ on both $\D\times 0$ and $\D\times 1$, and set $\K(F)(k_\D)=k_{\D'}$. A morphism of the form $d \times (0\to 1)$ in $\K(\D)$ is mapped to $F(d)\times (0\to 1)$, and the ones of the form $k_{\D} \to (d,1)$ are mapped to the unique morphisms $ k_{\D'} \to (F(d),1)$. This makes $\K$ into a functor. 
  
  \item We can identify $c\Sd^2\Delta[n]$ for $n\geq 1$ with $\K(c\Sd^2\del\Delta[n])$. As this category (and thus all of its subcategories) are posets, we will consider it either as a (partially) ordered set or as a category whenever convenient without further mentioning. For this, we use the description of $c\Sd^2\Delta[n]$ from Section 3 of \cite{FiorePaoli}. The objects of $c\Sd^2\Delta[n]$ are strictly increasing sequences $v_0\subsetneq \ldots \subsetneq v_{m}$, $m\geq 0$, of non-empty subsets of the set $\underline{n}$. The ordering is given as follows: The sequence $v_0\subsetneq \ldots \subsetneq v_{m}$ is less or equal ($\leq$) to $w_0 \subsetneq \ldots \subsetneq w_{l}$ iff the set $\{v_0, \ldots, v_m\}$ is contained in the set $\{w_0, \ldots, w_{l}\}$. The subposet $c\Sd^2\del\Delta[n]$ consists of all sequences $v_0\subsetneq \ldots \subsetneq v_{m}$ with $v_m\neq \un$. Note that any other element of $c\Sd^2\Delta[n]$ is either of the form 
 $v_0\subsetneq \ldots \subsetneq v_{m}\subsetneq \un$ with $v_0\subsetneq \ldots \subsetneq v_{m}$ in $c\Sd^2\del\Delta[n]$ or a sequence consisting of the single subset $\un$. We identify the latter with $k$ and the former with $c\Sd^2\del\Delta[n] \times 1$ in $\K(c\Sd^2\del\Delta[n])$. 
 The pictures for the case $n=1$ and $n=2$ might illustrate the situation. 
\begin{center}
  \begin{tikzpicture}[scale=0.8]
   \coordinate (a1) at (0,0);
      \coordinate (a2) at (2,0);
         \coordinate (a3) at (4,0);
            \coordinate (a4) at (6,0);
               \coordinate (a5) at (8,0);

\draw[fill, red](a1) circle (2pt)node(c1){} node[below] (b1){$0$};
\draw[fill, blue](a2) circle (2pt) node(c2){} node[below] (b2){$0\subsetneq \{0,1\}$};
\draw[fill](a3) circle (2pt) node(c3){} node[below] (b3){$\{0,1\}$};
\draw[fill, blue](a4) circle (2pt) node(c4){} node[below] (b4){$1\subsetneq \{0,1\}$};
\draw[fill, red](a5) circle (2pt)node(c5){} node[below](b5){$1$};

\draw[-stealth] (c1)--(c2);
\draw[-stealth] (c3)--(c2);
\draw[-stealth] (c3)--(c4);
\draw[-stealth] (c5)--(c4);

  \end{tikzpicture}
 \end{center}
and 
\begin{center}
  \begin{tikzpicture}
 \begin{scope}
 \coordinate (a1) at (0,0); 
 \coordinate (a2) at (7,0);
 \coordinate (a3) at (3.5,6.062);
 
\draw[fill, red](barycentric cs:a1=1,a2=0,a3=0) circle (2pt)node (b1){};
\draw[fill, red](barycentric cs:a1=0,a2=1,a3=0) circle (2pt)node (b2){};
\draw[fill, red](barycentric cs:a1=0,a2=0,a3=1.0) circle (2pt)node (b3){};
   
\draw[fill] (barycentric cs:a1=1,a2=1,a3=1) circle(2pt) node (b123){};

\draw[fill, red](barycentric cs:a1=1,a2=1) circle (2pt) node (b12){};
\draw[fill, red ](barycentric cs:a2=1,a3=1) circle (2pt) node (b23){};
\draw[fill, red](barycentric cs:a1=1,a3=1) circle (2pt) node (b13){};

\draw[fill, red](barycentric cs:b1=1,b12=1) circle (2pt) node (b112){};
\draw[fill, red](barycentric cs:b2=1,b12=1) circle (2pt) node (b212){};
\draw[fill, red](barycentric cs:b1=1,b13=1) circle (2pt) node (b113){};

\draw[fill, red](barycentric cs:b3=1,b13=1) circle (2pt) node (b313){};
\draw[fill, red](barycentric cs:b2=1,b23=1) circle (2pt) node (b223){};
\draw[fill, red](barycentric cs:b3=1,b23=1) circle (2pt) node (b323){};


\draw[fill, blue](barycentric cs:a1=1,b123=1) circle (2pt) node (b1123){};
\draw[fill,blue](barycentric cs:a2=1,b123=1) circle (2pt) node (b2123){};
\draw[fill, blue](barycentric cs:a3=1,b123=1) circle (2pt) node (b3123){};

\draw[fill,blue](barycentric cs:b12=1,b123=1) circle (2pt) node (b12123){};
\draw[fill, blue](barycentric cs:b23=1,b123=1) circle (2pt) node (b23123){};
\draw[fill, blue](barycentric cs:b13=1,b123=1) circle (2pt) node (b13123){};

\draw[fill, blue](barycentric cs:b1=1,b12=1,b123=1) circle (2pt) node (b112123){};
\draw[fill, blue](barycentric cs:b2=1,b23=1,b123=1) circle (2pt) node (b223123){};
\draw[fill, blue](barycentric cs:b3=1,b13=1,b123=1) circle (2pt) node (b313123){};

\draw[fill, blue](barycentric cs:b2=1,b12=1,b123=1) circle (2pt) node (b212123){};
\draw[fill, blue](barycentric cs:b3=1,b23=1,b123=1) circle (2pt) node (b323123){};
\draw[fill, blue](barycentric cs:b1=1,b13=1,b123=1) circle (2pt) node (b113123){};

\draw[-stealth, red] (b1)--(b112);
\draw[-stealth, red] (b1)--(b113);
\draw[-stealth, red] (b2)--(b212);
\draw[-stealth, red] (b2)--(b223);
\draw[-stealth, red] (b3)--(b313);
\draw[-stealth,red] (b3)--(b323);
\draw[-stealth] (b1)--(b1123);
\draw[-stealth] (b2)--(b2123);
\draw[-stealth] (b3)--(b3123);

\draw[-stealth] (b123)--(b1123);
\draw[-stealth] (b123)--(b2123);
\draw[-stealth] (b123)--(b3123);
\draw[-stealth] (b123)--(b12123);
\draw[-stealth] (b123)--(b23123);
\draw[-stealth] (b123)--(b13123);

\draw[-stealth] (b12)--(b12123);
\draw[-stealth] (b23)--(b23123);
\draw[-stealth] (b13)--(b13123);
\draw[-stealth, red] (b12)--(b112);
\draw[-stealth, red] (b13)--(b113);
\draw[-stealth, red] (b12)--(b212);
\draw[-stealth, red] (b23)--(b223);
\draw[-stealth, red] (b13)--(b313);
\draw[-stealth, red] (b23)--(b323);

\draw[-stealth] (b223)--(b223123);
\draw[-stealth, blue] (b2123)--(b223123);
\draw[-stealth, blue] (b23123)--(b223123);

 \draw[-stealth] (b212)--(b212123);
 \draw[-stealth, blue] (b2123)--(b212123);
 \draw[-stealth, blue] (b12123)--(b212123);
%
 \draw[-stealth] (b112)--(b112123);
 \draw[-stealth, blue] (b1123)--(b112123);
 \draw[-stealth, blue] (b12123)--(b112123);

\draw[-stealth] (b113)--(b113123);
\draw[-stealth, blue] (b1123)--(b113123);
\draw[-stealth, blue] (b13123)--(b113123);

\draw[-stealth] (b313)--(b313123);
\draw[-stealth, blue] (b3123)--(b313123);
\draw[-stealth, blue] (b13123)--(b313123);

\draw[-stealth] (b323)--(b323123);
\draw[-stealth, blue] (b3123)--(b323123);
\draw[-stealth, blue] (b23123)--(b323123);
\end{scope}
\end{tikzpicture}
\end{center}

\noindent Here, the red part is identified with $c\Sd^2\del\Delta[n] \times 0$, the blue part with $c\Sd^2\del\Delta[n] \times 1$ and the black point is the additional point $k$. Observe that we usually leave out the maps obtained as compositions of displayed maps in our pictures.

\item One can view $\K$ as a homotopically correct cone functor since for every category $\DD$, the nerve $N\KK(\DD)$ is contractible and $\DD \to \KK(\DD)$ is a Dwyer morphism. Dwyer morphisms form the cofibrations in a cofibration category structure on $\Cat$ as essentially proven in \cite{ThomasonCat} and observed in \cite[Section 1.4.5]{Karol}. Note that $\DD \to \DD^{\vartriangleleft}$ is not a Dwyer morphism.
 \end{enumerate}
\end{remark}

First, we show that we can use $\K$ to reformulate the fibrancy of $\Ex^2N \D$ for a category $\D$. Note to that purpose that by adjointness $\Ex^2N\D$ is fibrant if and only if $\DD$ has the lifting property with respect to all inclusions $c\Sd^2\Lambda^i[n] \to c\Sd^2\Delta[n]$.
\begin{Lemma}\label{CrucialLemma}
 For a category $\D$, to have a lifting property with respect to all inclusions $c\Sd^2 \Lambda^i[n] \to c\Sd^2 \Delta[n]$ is equivalent to having the lifting property with respect to all inclusions $c\Sd^2 \Lambda^n[n] \to \K(c\Sd^2 \Lambda^n[n])$. 
\end{Lemma}

\begin{pf}
 First, observe that for any $0\leq i\leq n$, there is an automorphism of the category $c\Sd^2\Delta[n]$ mapping the subcategory $c\Sd^2 \Lambda^i[n]$ isomorphically to $c\Sd^2 \Lambda^n[n]$. Therefore, it is enough to consider $i=n$. 
 
 Next, we observe that $\K(c\Sd^2 \Lambda^n[n])$ is isomorphic to a full subposet $\mathcal{P}$ of $c\Sd^2 \Delta[n]$. This follows from Remark \ref{PropertiesK} as $c\Sd^2\Lambda^n[n]$ is a subposet of $c\Sd^2\del\Delta[n]$. More explicitly, the subposet $c\Sd^2\Lambda^n[n]$ of $c\Sd^2\Delta[n]$ consists of all those sequences $v_0\subsetneq \ldots \subsetneq v_{m}$ for which $v_m \neq \un$ and $v_m\neq \underline{n-1}$. The subposet $\mathcal{P}$ of $c\Sd^2\Delta[n]$ contains all sequences $v_0\subsetneq \ldots \subsetneq v_{m}$ in $c\Sd^2\Lambda^n[n]$, for each such sequence also the sequence $v_0\subsetneq \ldots \subsetneq v_{m}\subsetneq \un$, and finally the sequence consisting only of $\un$ (corresponding to $k\in\K(c\Sd^2 \Lambda^n[n]))$. 
 
This implies immediately that having the lifting property for all morphisms $c\Sd^2 \Lambda^i[n] \to c\Sd^2 \Delta[n]$ implies the lifting property for all morphisms $c\Sd^2 \Lambda^n[n] \to \K(c\Sd^2 \Lambda^n[n])$. 
 
For the other implication, it is enough to show that each morphism defined on $\mathcal{P}$ can be extended to be defined on all of $c\Sd^2 \Delta[n]$. We will give a retraction for the inclusion of $\mathcal{P}$ into $c\Sd^2\Delta[n]$, i.e., an ordering-preserving map $c\Sd^2\Delta[n] \to \mathcal{P}$, which is identity on $\mathcal{P}$. This will complete the proof. Observe that the only objects of $c\Sd^2 \Delta[n]$ which are not in $\mathcal{P}$ are given by sequences in which $\underline{n}\setms n$ occurs; more precisely, these are the sequences $\un \setms n$, $\un \setms n \subsetneq \un$, and $w_0 \subsetneq \ldots \subsetneq w_{l} \subsetneq \un \setms n $ and $w_0 \subsetneq \ldots \subsetneq w_{l} \subsetneq \un \setms n \subsetneq \un$, where in the last two cases, $w_0 \subsetneq \ldots \subsetneq w_{l}$ is a sequence of non-empty subsets of $\un \setms n$.

The map $r\colon c\Sd^2\Delta[n] \to \mathcal{P}$ is described as follows:
\begin{eqnarray*}
 A\mapsto \begin{cases}
           A, &\mbox{ if } A \in \mathcal{P},\\
           \un, & \mbox{ if } A=\un \setms n \mbox{ or } \un \setms n \subsetneq \un, \\
           w_0 \subsetneq w_1\subsetneq \ldots \subsetneq w_l \subsetneq \un, &\mbox{ if } A=(w_0 \subsetneq w_1\subsetneq \ldots \subsetneq w_l \subsetneq \un \setms n) \mbox{ or }\\
           &A=(w_0 \subsetneq w_1\subsetneq \ldots \subsetneq w_l \subsetneq \un \setms n \subsetneq \un).
          \end{cases}
\end{eqnarray*}

Note that the assignment above covers all cases. Furthermore, the map takes only values in $\mathcal{P}$ and is by definition identity on $\mathcal{P}$. So the only thing to check is that $r$ is order-preserving. Note that the only changes to a sequence $A$ under $r$ is deleting the entry $\unn$ whenever it is present and if it was and $\un$ was not, adding $\un$. Now given $A\lneqq B$, if $A$ does not contain $\unn$, then it remains fixed under $r$ and deleting $\unn$ from $B$ or adding $\un$ to it does not change the relation, so in this case $r(A)\leq r(B)$. If $A$ does contain $\unn$, then so does $B$, and then both $r(A)$ and $r(B)$ do not contain $\unn$ and contain $\un$, while all the other entries remained unchanged, so we have again $r(A) \leq r(B)$, proving that $r$ is order-preserving. This completes the proof of the lemma. 
\end{pf}

\begin{example}
 We want to illustrate the procedure in the proof of the last lemma for the case of the 2-simplex. By definition, the poset $\mathcal{P}$ can be drawn as follows:
\begin{center}
 \begin{tikzpicture}[scale=0.8]
 \begin{scope}
 \coordinate (a1) at (0,0); 
 \coordinate (a2) at (7,0);
 \coordinate (a3) at (3.5,6.062);
 
\draw[fill](barycentric cs:a1=1,a2=0,a3=0) circle (2pt)node (b1){};
\draw[fill](barycentric cs:a1=0,a2=1,a3=0) circle (2pt)node (b2){};
\draw[fill](barycentric cs:a1=0,a2=0,a3=1.0) circle (2pt)node (b3){};
   
\draw[fill] (barycentric cs:a1=1,a2=1,a3=1) circle(2pt) node (b123){};

\draw[fill](barycentric cs:a2=1,a3=1) circle (2pt) node (b23){};
\draw[fill](barycentric cs:a1=1,a3=1) circle (2pt) node (b13){};

\draw[fill](barycentric cs:b1=1,b13=1) circle (2pt) node (b113){};

\draw[fill](barycentric cs:b3=1,b13=1) circle (2pt) node (b313){};
\draw[fill](barycentric cs:b2=1,b23=1) circle (2pt) node (b223){};
\draw[fill](barycentric cs:b3=1,b23=1) circle (2pt) node (b323){};


\draw[fill, blue](barycentric cs:a1=1,b123=1) circle (2pt) node (b1123){};
\draw[fill, blue](barycentric cs:a2=1,b123=1) circle (2pt) node (b2123){};
\draw[fill, blue](barycentric cs:a3=1,b123=1) circle (2pt) node (b3123){};

\draw[fill, blue](barycentric cs:b23=1,b123=1) circle (2pt) node (b23123){};
\draw[fill, blue](barycentric cs:b13=1,b123=1) circle (2pt) node (b13123){};

\draw[fill, blue](barycentric cs:b2=1,b23=1,b123=1) circle (2pt) node (b223123){};
\draw[fill, blue](barycentric cs:b3=1,b13=1,b123=1) circle (2pt) node (b313123){};

\draw[fill, blue](barycentric cs:b3=1,b23=1,b123=1) circle (2pt) node (b323123){};
\draw[fill, blue](barycentric cs:b1=1,b13=1,b123=1) circle (2pt) node (b113123){};

\draw[-stealth] (b1)--(b113);
\draw[-stealth] (b2)--(b223);
\draw[-stealth] (b3)--(b313);
\draw[-stealth] (b3)--(b323);
\draw[-stealth] (b1)--(b1123);
\draw[-stealth] (b2)--(b2123);
\draw[-stealth] (b3)--(b3123);

\draw[-stealth] (b123)--(b1123);
\draw[-stealth] (b123)--(b2123);
\draw[-stealth] (b123)--(b3123);
\draw[-stealth] (b123)--(b23123);
\draw[-stealth] (b123)--(b13123);

\draw[-stealth] (b23)--(b23123);
\draw[-stealth] (b13)--(b13123);
\draw[-stealth] (b13)--(b113);
\draw[-stealth] (b23)--(b223);
\draw[-stealth] (b13)--(b313);
\draw[-stealth] (b23)--(b323);

\draw[-stealth] (b223)--(b223123);
\draw[-stealth, blue] (b2123)--(b223123);
\draw[-stealth, blue] (b23123)--(b223123);

%

\draw[-stealth] (b113)--(b113123);
\draw[-stealth, blue] (b1123)--(b113123);
\draw[-stealth, blue] (b13123)--(b113123);

\draw[-stealth] (b313)--(b313123);
\draw[-stealth, blue] (b3123)--(b313123);
\draw[-stealth, blue] (b13123)--(b313123);

\draw[-stealth] (b323)--(b323123);
\draw[-stealth, blue] (b3123)--(b323123);
\draw[-stealth, blue] (b23123)--(b323123);
\end{scope}
 
 \end{tikzpicture}
 \end{center}
 
Moreover, we highlighted in blue the part of it identified with $c\Sd^2\Lambda^2[2] \times 1$ as part of $\K(c\Sd^2 \Lambda^2[2])$; the barycenter of the large triangle is identified with the additional point.

 Now the retraction $r$ maps the red dots here to the corresponding blue dot and the red arrows to identities of this object. 
 
  \begin{center}
 \begin{tikzpicture}[scale=0.8]
 \begin{scope}
 \coordinate (a1) at (0,0); 
 \coordinate (a2) at (7,0);
 \coordinate (a3) at (3.5,6.062);
 
\draw[fill](barycentric cs:a1=1,a2=0,a3=0) circle (2pt)node (b1){};
\draw[fill](barycentric cs:a1=0,a2=1,a3=0) circle (2pt)node (b2){};
\draw[fill](barycentric cs:a1=0,a2=0,a3=1.0) circle (2pt)node (b3){};
   
\draw[fill, blue] (barycentric cs:a1=1,a2=1,a3=1) circle(2pt) node (b123){};

\draw[fill, red](barycentric cs:a1=1,a2=1) circle (2pt) node (b12){};
\draw[fill](barycentric cs:a2=1,a3=1) circle (2pt) node (b23){};
\draw[fill](barycentric cs:a1=1,a3=1) circle (2pt) node (b13){};

\draw[fill, red](barycentric cs:b1=1,b12=1) circle (2pt) node (b112){};
\draw[fill, red](barycentric cs:b2=1,b12=1) circle (2pt) node (b212){};
\draw[fill](barycentric cs:b1=1,b13=1) circle (2pt) node (b113){};

\draw[fill](barycentric cs:b3=1,b13=1) circle (2pt) node (b313){};
\draw[fill](barycentric cs:b2=1,b23=1) circle (2pt) node (b223){};
\draw[fill](barycentric cs:b3=1,b23=1) circle (2pt) node (b323){};


\draw[fill, blue](barycentric cs:a1=1,b123=1) circle (2pt) node (b1123){};
\draw[fill,blue](barycentric cs:a2=1,b123=1) circle (2pt) node (b2123){};
\draw[fill](barycentric cs:a3=1,b123=1) circle (2pt) node (b3123){};

\draw[fill,red](barycentric cs:b12=1,b123=1) circle (2pt) node (b12123){};
\draw[fill](barycentric cs:b23=1,b123=1) circle (2pt) node (b23123){};
\draw[fill](barycentric cs:b13=1,b123=1) circle (2pt) node (b13123){};

\draw[fill,red](barycentric cs:b1=1,b12=1,b123=1) circle (2pt) node (b112123){};
\draw[fill](barycentric cs:b2=1,b23=1,b123=1) circle (2pt) node (b223123){};
\draw[fill](barycentric cs:b3=1,b13=1,b123=1) circle (2pt) node (b313123){};

\draw[fill,red](barycentric cs:b2=1,b12=1,b123=1) circle (2pt) node (b212123){};
\draw[fill](barycentric cs:b3=1,b23=1,b123=1) circle (2pt) node (b323123){};
\draw[fill](barycentric cs:b1=1,b13=1,b123=1) circle (2pt) node (b113123){};

\draw[-stealth] (b1)--(b112);
\draw[-stealth] (b1)--(b113);
\draw[-stealth] (b2)--(b212);
\draw[-stealth] (b2)--(b223);
\draw[-stealth] (b3)--(b313);
\draw[-stealth] (b3)--(b323);
\draw[-stealth] (b1)--(b1123);
\draw[-stealth] (b2)--(b2123);
\draw[-stealth] (b3)--(b3123);

\draw[-stealth] (b123)--(b1123);
\draw[-stealth] (b123)--(b2123);
\draw[-stealth] (b123)--(b3123);
\draw[-stealth, red] (b123)--(b12123);
\draw[-stealth] (b123)--(b23123);
\draw[-stealth] (b123)--(b13123);

\draw[-stealth, red] (b12)--(b12123);
\draw[-stealth] (b23)--(b23123);
\draw[-stealth] (b13)--(b13123);
\draw[-stealth] (b12)--(b112);
\draw[-stealth] (b13)--(b113);
\draw[-stealth] (b12)--(b212);
\draw[-stealth] (b23)--(b223);
\draw[-stealth] (b13)--(b313);
\draw[-stealth] (b23)--(b323);

\draw[-stealth] (b223)--(b223123);
\draw[-stealth] (b2123)--(b223123);
\draw[-stealth] (b23123)--(b223123);

 \draw[-stealth, red] (b212)--(b212123);
 \draw[-stealth, red] (b2123)--(b212123);
 \draw[-stealth] (b12123)--(b212123);
%
 \draw[-stealth, red] (b112)--(b112123);
 \draw[-stealth, red] (b1123)--(b112123);
 \draw[-stealth] (b12123)--(b112123);

\draw[-stealth] (b113)--(b113123);
\draw[-stealth] (b1123)--(b113123);
\draw[-stealth] (b13123)--(b113123);

\draw[-stealth] (b313)--(b313123);
\draw[-stealth] (b3123)--(b313123);
\draw[-stealth] (b13123)--(b313123);

\draw[-stealth] (b323)--(b323123);
\draw[-stealth] (b3123)--(b323123);
\draw[-stealth] (b23123)--(b323123);
\end{scope}
 \end{tikzpicture}
 \end{center}
 
 The ``horizontal'' arrows are sent to corresponding arrows between the blue points (and the vertices $1$ and $2$). 
\end{example}

\begin{Definition}Let $\DD$ be a category. A functor 
\begin{eqnarray*}
 \Phi\colon \Fun(\DD, \WW) \to \Fun(\K(\DD), \WW)
\end{eqnarray*}
 is called an \textbf{extension functor} (along $i$) if $i^*\circ \Phi = \id$ for the inclusion $i\colon \DD = \DD\times 0 \to \K(\DD)$. We will consider the following two properties of extension functors. 
 
\begin{description}[style=multiline, labelwidth=1.5cm]
    \item[\namedlabel{F1}{\textbf{(Cof)}}] An extension functor $\Phi$ is said to fulfill \ref*{F1} if for every functor $\alpha\colon \DD \to \WW$ and for every object $x\in \DD$, the morphism 
    \[\Phi(\alpha)(x\times (0\to 1))
    \]
 is a cofibration. 
    \item[\namedlabel{F3}{\textbf{(Lim)}}] An extension functor $\Phi$ is said to fulfill \ref*{F3} if for every functor $\alpha \colon \DD\to \WW$, the restriction $\Phi(\alpha)|_{\D^{\vartriangleleft}}$ is a limit diagram for $\Phi(\alpha)|_{\DD\times 1}$.
    
\end{description}

 \end{Definition}

The following auxiliary lemma is often convenient to show \ref{F3} and is not hard to prove.  
\begin{Lemma}\label{KLimit}
 Let $I$ be a category and $\beta \colon\K(I)\to \mathcal{D}$ some functor. Assume that $B_0=\lim_{I\times 0}\beta|_{I\times 0}$  and $B_1=\lim_{I\times 1}\beta|_{I\times 1}$ exist and denote by $g\colon B_0\to B_1$ the induced map. Furthermore, the compatible maps $\beta(k) \to \beta(i\times 1)$ for all objects $i\in I$ induce a map $h\colon\beta(k)\to B_1$. Assume the diagram
 \begin{eqnarray*}
  B_0\xrightarrow{g} B_1 \xleftarrow{h} \beta(k)
 \end{eqnarray*}
has a pullback $P$. Then the diagram $\beta \colon\K(I)\to \mathcal{D}$ has a limit and $P\cong \lim_{\K(I)} \beta$. Moreover, the projections from $P$ to objects of $\K(I)$ factor through $B_0$, $B_1$ or $\beta(k)$, respectively. 
\end{Lemma}

Recall that our goal is to construct an extension functor along \[c\Sd^2\Lambda^{n+1}[n+1] \to \K(c\Sd^2\Lambda^{n+1}[n+1]).\]
This will be glued from extension functors along
\[c\Sd^2\Delta[n] \to \K(c\Sd^2\Delta[n]),\]
 satisfying some boundary condition. These in turn will (inductively) be defined via extension functors along 
 \[c\Sd^2\Lambda^n[n] \to \K(c\Sd^2\Lambda^n[n]).\]
 As noted above, we can view the poset $\K(c\Sd^2\Lambda^n[n])$ as a subposet of $c\Sd^2\Delta[n]$ and our plan is first to define a nice extension functor along
\[\K(c\Sd^2\Lambda^{n}[n]) \to \K(\K(c\Sd^2\Lambda^{n}[n])).\]
 This works for any category $\DD$ instead of $c\Sd^2\Lambda^{n}[n]$. To distinguish the two ``directions'' of applying $\K$, we will denote them by $\K_h$ for horizontal and $\K_v$ for vertical one, so that we write $\K_v(\K_h(\DD))$ instead of $\K(\K(\DD))$. We will write short $k_v$ for $k_{\K_h(\DD)}$. 
 
 \begin{Lemma}\label{DoubleK}
  Let $\D$ be a category with a given extension functor 
  \begin{eqnarray*}
   \Phi\colon \Fun(\DD, \WW) \to \Fun(\K(\DD), \WW).
  \end{eqnarray*}
 Assume that $\Phi$ satisfies \ref{F1} and \ref{F3}. Then there is an extension functor 
\[\Phi' \colon \Fun(\K_h(\D), \WW) \to \Fun(\K_v(\K_h(\D)), \WW)\]
 satisfying \ref{F1} and \ref{F3} and a natural transformation 
\[\varepsilon \colon \K_v(i)^*\circ\Phi' \Rightarrow \Phi\circ i^*,\] 
where $i$ now denotes the inclusion $\D \to \K_h(\D)$, with the following property: 
For any $\alpha \colon \K_h(\D) \to \WW$, we have $\Phi(\alpha|_{\D})|_{\D\times (0\to 1)_v}=\Phi'(\alpha)|_{\D\times (0\to 1)_v}$ and $\varepsilon$ induces the identity between them and $\varepsilon_{k_v}(\alpha)$ is a fibration. In other words: $\Phi'(\alpha)$ agrees with $\Phi(\alpha|_\DD)$ where possible except at $k_v$ and we have a compatible fibration $\Phi'(\alpha)(k_v) \to \Phi(\alpha|_\DD)(k_v)$. 
 \end{Lemma}

 \begin{pf}
   To construct $\Phi'$, we proceed in several steps. First, we construct out of $\Phi$ for each given $\alpha\colon \K_h(\D)\to \WW$ an intermediate extension $\alpha'$, then we improve it to an extension $\alpha''$ satisfying \ref{F1} and then we show it satisfies also $\ref{F3}$. The extension $\alpha''$ will be our $\Phi'(\alpha)$. 
 \begin{enumerate}[label=\textbf{Step \arabic{enumi}:}, ref=Step \arabic{enumi}]
  \item \label{LStep1}
   We can apply $\Phi$ to $\alpha|_{\D \times 0_h}$ and to $\alpha|_{\D \times 1_h}$. Since the images of $(0\to 1)_h$ induce under $\alpha$ a natural transformation between the two diagrams, we obtain a natural transformation 
	\[\Phi(\alpha|_{\D \times 0_h}) \Rightarrow \Phi(\alpha|_{\D \times 1_h}) .\]
	
	Define now $\alpha' \colon \K_v\K_h(\D) \to \mathcal{W}$ to be $\Phi(\alpha|_{\D \times 0_h})$ on $K_v(\D \times 0_h)$, furthermore $\Phi(\alpha|_{\D \times 1_h})$ on $\D \times 1_h \times (0\to 1)_v$ and the natural transformation discussed above on $\D \times (0\to 1)_h \times \{0,  1\}_v$. Moreover, define $\alpha'((k_h, 1_v))=\Phi(\alpha|_{\D \times 1_h})(k)$ and 
 \begin{eqnarray*}
  \alpha'(k_v\to (k_h,1_v))=\Phi(\alpha|_{\D \times 0_h} \Rightarrow \alpha|_{\D \times 1_h})(k)                
\end{eqnarray*}
 We still need to define $\alpha'(k_h \times (0\to 1)_v)$. Observe that we have compatible maps from $\alpha'((k_h, 0_v))=\alpha(k_h)$ to each object in 
 \begin{eqnarray*}
  \alpha'(\D \times 1_h \times 0_v)=\alpha(\D \times 1_h \times 0_v).
 \end{eqnarray*}
 We can compose these maps with maps given by $\alpha'(\D \times 1_h \times (0 \to 1)_v)$ to obtain compatible maps from $\alpha(k_h)$ to $\alpha'(\D \times 1_h \times 1_v)$. Since by assumption the functor $\Phi$ satisfies the property \ref{F3}, the limit of this last diagram is $\Phi(\alpha|_{\D \times 1_h})(k)=\alpha'((k_h, 1_v))$. Thus, there is a unique morphism $\alpha'((k_h, 0_v))\to \alpha'((k_h, 1_v))$ making all the relevant diagrams commute, so that we can define $\alpha'(k_h \times (0\to 1)_v)$ to be this morphism. Thus, we have extended $\alpha$ to $\K_v\K_h(\D)$. 
 
 \item \label{LStep2} Observe that all the maps of the form $x \times (0\to 1)_v$ for some object $x$ of $\K_h(\D)$ except for $x=k_h$ are mapped by $\alpha'$ to cofibrations due to property \ref{F1} of $\Phi$ and the construction in \ref{LStep1}. Since we need to fulfill \ref{F1} again, we want to replace $\alpha'(k_h\times (0\to 1)_v)$ by a cofibration. We start by functorially factorizing this morphism into a cofibration $g_1\colon A_1\to A_2$ followed by a fibration $g_2\colon A_2\to A_3$, so $\alpha'(k_h\times (0\to 1)_v)=g_2\circ g_1$. Now define $\alpha''\colon \K_v\K_h(\D)\to \mathcal{W}$ to coincide with $\alpha'$ everywhere except on objects $(k_h, 1_v)$ and $k_v$ and on morphisms starting or ending in these objects. Set $\alpha''((k_h, 1_v))$ to be $A_2$ and define $\alpha''(k_h \times (0\to 1)_v)=g_1$. For any map starting in $(k_h, 1_v)$, we precompose its image under $\alpha'$ with $g_2$ to obtain its image under $\alpha''$. Last, we have the problem that there is in general no map from $\alpha'
(
k_v)$ to $A_2$. But since we have a cospan 
\[\alpha'(k_v) \to \alpha'((k_h,1_v))=A_3 \xleftarrow{g_2} A_2\] and $g_2$ is a fibration, we can define $\alpha''(k_v)$ as its pullback. As a pullback of a fibration, the morphism $\alpha''(k_v) \to \alpha'(k_v)$ is a fibration. Note that since there are only morphisms starting in $k_v$ in $\K_v\K_h(\D)$, we can simply precompose the image of each such map under $\alpha'$ with the pullback projection morphism $\alpha''(k_v) \to \alpha'(k_v)$ to obtain the corresponding images unter $\alpha''$ (except for the morphism to $A_2$). This now defines a functor $\alpha''\colon \K_v\K_h(\D) \to \mathcal{W}$, and observe that by construction now all morphisms of the form $\alpha''(x\times (0\to 1)_v)$ are cofibrations. We set $\Phi'(\alpha)=\alpha''$ and obtain an extension functor that, due to modification in this step, satifies \ref{F1}. It is a functor since pullbacks and factorizations are functorial. 

Note that by definition, $\Phi'(\alpha)$ and $\Phi(\alpha|_{\D\times 0_h})$ coincide when restricted to $\D \times 0_h \times (0\to 1)_v$. Moreover, by 
construction, the map 
 \begin{eqnarray*}
  \Phi'(\alpha)(k_v)=\alpha''(k_v)\to \Phi(\alpha|_{\D\times 0_h})(k_v)=\alpha'(k_v)
 \end{eqnarray*}
 is a fibration.  Altogether, we already have constructed a natural transformation $\varepsilon$ as required. 
 
 \item \label{LStep3} To prove that $\Phi'$ satisfies \ref{F3}, we want to prove that $\alpha''(k_v)$ with the corresponding maps is a limit of $\alpha''|_{\K_h(\D) \times 1_v}$. 
 
 This is a consequence of Lemma \ref{KLimit}. Indeed, recall that on both $\D \times 0_h\times 1_v$ and $\D \times 1_h\times 1_v$, the functor $\alpha''$ coincides with $\alpha'$ and is given by 
 \begin{eqnarray*}
  \Phi(\alpha|_{\D \times 0_h})|_{\D \times 0_h\times 1_v} \mbox{ and }  \Phi(\alpha|_{\D\times 1_h})|_{\D \times 1_h\times 1_v},
 \end{eqnarray*}
respectively. By Property \ref{F3} of $\Phi$, both of these diagrams have a limit, namely $\alpha'(k_v)$ and $\alpha'((k_h, 1_v))$, respectively; the map 
\begin{eqnarray*}
 \alpha'(k_v \to (k_h, 1_v))
\end{eqnarray*}
 is exactly the one induced by maps of the form $x \times (0\to 1)_h$. Moreover, the map $g_2\colon \alpha''((k_h, 1_v))\to \alpha'((k_h, 1_v))$ is the induced map to the limit. Thus, by Lemma \ref{KLimit}, $\alpha''(k_v)$ is the limit of the diagram $\alpha''|_{\K_h(\D) \times 1_v}$. This completes the proof of the lemma.
 \end{enumerate}
  
 \end{pf}

Next, we want to prove that the category of weak equivalences of a given partial model category has the lifting property with respect to all inclusions $c\Sd^2 \Lambda^n[n] \to \K(c\Sd^2 \Lambda^n[n])$. More precisely, we will inductively prove the following statement: 
\begin{Theorem} \label{MainTheorem}
 Let $(\mathcal{M}, \mathcal{W})$ be a partial model category. Then for each $n$, there are extension functors 
 \begin{eqnarray*}
  \Phi_n \colon  \Fun(c\Sd^2 \Lambda^n[n], \mathcal{W})\to \Fun(\K(c\Sd^2 \Lambda^n[n]), \mathcal{W})
 \end{eqnarray*}
 and
 \begin{eqnarray*}
  \Psi_n\colon \Fun(c\Sd^2 \Delta[n], \mathcal{W})\to \Fun(\K(c\Sd^2 \Delta[n]), \mathcal{W})
 \end{eqnarray*}
 fulfilling \ref{F1} and \ref{F3} and the following additional boundary conditions:

 \begin{description}
  \item[\namedlabel{F4}{\textbf{(F1)}}] \label{F4}The order on the vertices of an $n$-simplex in the boundary of an ${(n+1)}$-simplex gives a distinguished way to identify it with the standard $\Delta ^n$. With this identification, we require that for any 
  \begin{eqnarray*}
  \alpha \colon c\Sd^2 \Lambda^{n+1}[n+1] \to \mathcal{W} 
  \end{eqnarray*}
 and for all $i\neq n$ we have  
  \begin{eqnarray*}
  \Phi_{n+1}(\alpha) |_{c\Sd^2 d_{i}\Delta^{n+1} \times (0\to 1)} =\Psi_n(\alpha|_{c\Sd^2 d_{i}\Delta^{n+1}})|_{c\Sd^2 d_{i}\Delta^{n+1} \times (0\to 1)} 
  \end{eqnarray*}
 and a compatible morphism 
 \begin{eqnarray*}
  \Phi_{n+1}(\alpha)(k)\to \Psi_n(\alpha|_{c\Sd^2 d_{i}\Delta^{n+1}})(k).  
 \end{eqnarray*}

\item[\namedlabel{F6}{\textbf{(F2)}}] \label{F6} Likewise, we require that for any $\alpha \colon c\Sd^2 \Delta[n] \to \WW$ we have
  \begin{eqnarray*}
  \Psi_n(\alpha) |_{c\Sd^2 \Lambda^n[n] \times (0\to 1)} =\Phi_n(\alpha|_{c\Sd^2 \Lambda^n[n]})|_{c\Sd^2 \Lambda^n[n] \times (0\to 1)} 
  \end{eqnarray*}
 and a compatible morphism 
 \begin{eqnarray*}
  \Psi_n(\alpha)(k)\to \Phi_n(\alpha|_{c\Sd^2 \Lambda^n[n]})(k).
 \end{eqnarray*}

\item[\namedlabel{F5}{\textbf{(F3)}}] \label{F5} Last, we require for any $\alpha \colon c\Sd^2 \Lambda^n[n] \to \mathcal{W}$ the morphism 
\[ \Phi_n(\alpha)(k) \to \Phi_n(\alpha)(\{n\}, 1) \]
 to be fibration. Here, $\{n\}$ stands for the length-$1$-chain of subsets of $\un$, corresponding to an object in $c\Sd^2\Lambda^n[n]$. Similarly, we want for each $\alpha \colon c\Sd^2 \Delta[n] \to \mathcal{W}$ the map 
\[\Psi_n(\alpha)(k) \to \Psi_n(\alpha)(\{n\}, 1)\]
 to be fibration. 
 \end{description}

\end{Theorem}

Before proving the theorem, we would like to remark the following: All we need to apply Lemma \ref{CrucialLemma} to deduce our main theorem is that $\Phi$ is an extension functor. All the other properties are just for the purposes of the induction.

\begin{pf}
As we already mentioned, the proof works by induction. We will explain first the plan of the proof before we give the details. 

Given $\Phi_{n}$ and $\Psi_{n-1}$, we construct first $\Psi_n$. Recall that we identified a part (denoted by $\mathcal{P}$ and highlighted blue in the picture below) of $c\Sd^2 \Delta[n]$ with $\K(c\Sd^2\Lambda^n[n])$. Having again two applications of $\K$ around now, we want to call this one ``horizontal'' and denote it by $\K_h$, while the other one is called vertical and denoted by $\K_v$.  To construct $\Psi_n$, we start with a functor $\alpha\colon c\Sd^2 \Delta[n] \to \mathcal{W}$ and extend its restriction $\alpha|_\mathcal{P}$ first to $\K_v(\mathcal{P})\subset \K_v(c\Sd^2 \Delta[n])$ using Lemma \ref{DoubleK}. This is the main point of \ref{Step1}. 
\begin{center}
 \begin{tikzpicture}[scale=0.8]
 \begin{scope}
 \coordinate (a1) at (0,0); 
 \coordinate (a2) at (7,0);
 \coordinate (a3) at (3.5,6.062);
 
\draw[fill, blue](barycentric cs:a1=1,a2=0,a3=0) circle (2pt)node (b1){};
\draw[fill, blue](barycentric cs:a1=0,a2=1,a3=0) circle (2pt)node (b2){};
\draw[fill, blue](barycentric cs:a1=0,a2=0,a3=1.0) circle (2pt)node (b3){};
   
\draw[fill, blue] (barycentric cs:a1=1,a2=1,a3=1) circle(2pt) node (b123){};

\draw[fill](barycentric cs:a1=1,a2=1) circle (2pt) node (b12){};
\draw[fill, blue](barycentric cs:a2=1,a3=1) circle (2pt) node (b23){};
\draw[fill, blue](barycentric cs:a1=1,a3=1) circle (2pt) node (b13){};

\draw[fill](barycentric cs:b1=1,b12=1) circle (2pt) node (b112){};
\draw[fill](barycentric cs:b2=1,b12=1) circle (2pt) node (b212){};
\draw[fill, blue](barycentric cs:b1=1,b13=1) circle (2pt) node (b113){};

\draw[fill, blue](barycentric cs:b3=1,b13=1) circle (2pt) node (b313){};
\draw[fill, blue](barycentric cs:b2=1,b23=1) circle (2pt) node (b223){};
\draw[fill, blue](barycentric cs:b3=1,b23=1) circle (2pt) node (b323){};


\draw[fill, blue](barycentric cs:a1=1,b123=1) circle (2pt) node (b1123){};
\draw[fill, blue](barycentric cs:a2=1,b123=1) circle (2pt) node (b2123){};
\draw[fill, blue](barycentric cs:a3=1,b123=1) circle (2pt) node (b3123){};

\draw[fill](barycentric cs:b12=1,b123=1) circle (2pt) node (b12123){};
\draw[fill, blue](barycentric cs:b23=1,b123=1) circle (2pt) node (b23123){};
\draw[fill, blue](barycentric cs:b13=1,b123=1) circle (2pt) node (b13123){};

\draw[fill](barycentric cs:b1=1,b12=1,b123=1) circle (2pt) node (b112123){};
\draw[fill, blue](barycentric cs:b2=1,b23=1,b123=1) circle (2pt) node (b223123){};
\draw[fill, blue](barycentric cs:b3=1,b13=1,b123=1) circle (2pt) node (b313123){};

\draw[fill](barycentric cs:b2=1,b12=1,b123=1) circle (2pt) node (b212123){};
\draw[fill, blue](barycentric cs:b3=1,b23=1,b123=1) circle (2pt) node (b323123){};
\draw[fill, blue](barycentric cs:b1=1,b13=1,b123=1) circle (2pt) node (b113123){};

\draw[-stealth] (b1)--(b112);
\draw[-stealth, blue] (b1)--(b113);
\draw[-stealth] (b2)--(b212);
\draw[-stealth, blue] (b2)--(b223);
\draw[-stealth, blue] (b3)--(b313);
\draw[-stealth, blue] (b3)--(b323);
\draw[-stealth, blue] (b1)--(b1123);
\draw[-stealth, blue] (b2)--(b2123);
\draw[-stealth, blue] (b3)--(b3123);

\draw[-stealth, blue] (b123)--(b1123);
\draw[-stealth, blue] (b123)--(b2123);
\draw[-stealth, blue] (b123)--(b3123);
\draw[-stealth] (b123)--(b12123);
\draw[-stealth, blue] (b123)--(b23123);
\draw[-stealth, blue] (b123)--(b13123);

\draw[-stealth] (b12)--(b12123);
\draw[-stealth, blue] (b23)--(b23123);
\draw[-stealth, blue] (b13)--(b13123);
\draw[-stealth] (b12)--(b112);
\draw[-stealth, blue] (b13)--(b113);
\draw[-stealth] (b12)--(b212);
\draw[-stealth, blue] (b23)--(b223);
\draw[-stealth, blue] (b13)--(b313);
\draw[-stealth, blue] (b23)--(b323);

\draw[-stealth, blue] (b223)--(b223123);
\draw[-stealth, blue] (b2123)--(b223123);
\draw[-stealth, blue] (b23123)--(b223123);

\draw[-stealth] (b212)--(b212123);
\draw[-stealth] (b2123)--(b212123);
\draw[-stealth] (b12123)--(b212123);

\draw[-stealth] (b112)--(b112123);
\draw[-stealth] (b1123)--(b112123);
\draw[-stealth] (b12123)--(b112123);

\draw[-stealth, blue] (b113)--(b113123);
\draw[-stealth, blue] (b1123)--(b113123);
\draw[-stealth, blue] (b13123)--(b113123);

\draw[-stealth, blue] (b313)--(b313123);
\draw[-stealth, blue] (b3123)--(b313123);
\draw[-stealth, blue] (b13123)--(b313123);

\draw[-stealth, blue] (b323)--(b323123);
\draw[-stealth, blue] (b3123)--(b323123);
\draw[-stealth, blue] (b23123)--(b323123);
\end{scope}
 
 \end{tikzpicture}
 \end{center}

We claim that we can consider the subdivided simplex $c\Sd^2\Delta[n]$ as glued from two pieces, one of which is $\mathcal{P}$ identified with $\K_h(c\Sd^2\Lambda^n[n])$. A part of $\mathcal{P}$ is constituted by $\K_h(c\Sd^2\partial d_n\Delta[n])$. Here, $c\Sd^2\del d_n\Delta[n] \times 0_h$ corresponds to sequences $v_0\subsetneq \ldots \subsetneq v_m$ of non-empty subsets of $\unn$ not containing $\unn$ itself. The $c\Sd^2\del d_n\Delta[n] \times 1_h$-part consists of sequences  $v_0\subsetneq \ldots \subsetneq v_m\subset \un$, where $v_0\subsetneq \ldots \subsetneq v_m$ is as in the last sentence. The additional point $k$ is identified with the barycenter $\un$. On the other hand, $\K_{h}(c \Sd^2\partial d_n\Delta[n])$ can be identified with a double subdivision of an $(n-1)$-simplex. In the case of $n=2$, $cSd^2 d_n\Delta[n]$ is marked in green (and red) and $c \Sd^2\partial d_n\Delta[n]$ in red in the picture below; the latter consists only of two $0$-simplices. 

\begin{center}
 \begin{tikzpicture}[scale=0.8]
 \begin{scope}
 \coordinate (a1) at (0,0); 
 \coordinate (a2) at (7,0);
 \coordinate (a3) at (3.5,6.062);
 
\draw[fill, red](barycentric cs:a1=1,a2=0,a3=0) circle (2pt)node (b1){};
\draw[fill, red](barycentric cs:a1=0,a2=1,a3=0) circle (2pt)node (b2){};
\draw[fill](barycentric cs:a1=0,a2=0,a3=1.0) circle (2pt)node (b3){};
   
\draw[fill] (barycentric cs:a1=1,a2=1,a3=1) circle(2pt) node (b123){};

\draw[fill, green](barycentric cs:a1=1,a2=1) circle (2pt) node (b12){};
\draw[fill](barycentric cs:a2=1,a3=1) circle (2pt) node (b23){};
\draw[fill](barycentric cs:a1=1,a3=1) circle (2pt) node (b13){};

\draw[fill, green](barycentric cs:b1=1,b12=1) circle (2pt) node (b112){};
\draw[fill, green](barycentric cs:b2=1,b12=1) circle (2pt) node (b212){};
\draw[fill](barycentric cs:b1=1,b13=1) circle (2pt) node (b113){};

\draw[fill](barycentric cs:b3=1,b13=1) circle (2pt) node (b313){};
\draw[fill](barycentric cs:b2=1,b23=1) circle (2pt) node (b223){};
\draw[fill](barycentric cs:b3=1,b23=1) circle (2pt) node (b323){};


\draw[fill](barycentric cs:a1=1,b123=1) circle (2pt) node (b1123){};
\draw[fill](barycentric cs:a2=1,b123=1) circle (2pt) node (b2123){};
\draw[fill](barycentric cs:a3=1,b123=1) circle (2pt) node (b3123){};

\draw[fill](barycentric cs:b12=1,b123=1) circle (2pt) node (b12123){};
\draw[fill](barycentric cs:b23=1,b123=1) circle (2pt) node (b23123){};
\draw[fill](barycentric cs:b13=1,b123=1) circle (2pt) node (b13123){};

\draw[fill](barycentric cs:b1=1,b12=1,b123=1) circle (2pt) node (b112123){};
\draw[fill](barycentric cs:b2=1,b23=1,b123=1) circle (2pt) node (b223123){};
\draw[fill](barycentric cs:b3=1,b13=1,b123=1) circle (2pt) node (b313123){};

\draw[fill](barycentric cs:b2=1,b12=1,b123=1) circle (2pt) node (b212123){};
\draw[fill](barycentric cs:b3=1,b23=1,b123=1) circle (2pt) node (b323123){};
\draw[fill](barycentric cs:b1=1,b13=1,b123=1) circle (2pt) node (b113123){};

\draw[-stealth, green] (b1)--(b112);
\draw[-stealth] (b1)--(b113);
\draw[-stealth, green] (b2)--(b212);
\draw[-stealth] (b2)--(b223);
\draw[-stealth] (b3)--(b313);
\draw[-stealth] (b3)--(b323);
\draw[-stealth] (b1)--(b1123);
\draw[-stealth] (b2)--(b2123);
\draw[-stealth] (b3)--(b3123);

\draw[-stealth] (b123)--(b1123);
\draw[-stealth] (b123)--(b2123);
\draw[-stealth] (b123)--(b3123);
\draw[-stealth] (b123)--(b12123);
\draw[-stealth] (b123)--(b23123);
\draw[-stealth] (b123)--(b13123);

\draw[-stealth] (b12)--(b12123);
\draw[-stealth] (b23)--(b23123);
\draw[-stealth] (b13)--(b13123);
\draw[-stealth, green] (b12)--(b112);
\draw[-stealth] (b13)--(b113);
\draw[-stealth, green] (b12)--(b212);
\draw[-stealth] (b23)--(b223);
\draw[-stealth] (b13)--(b313);
\draw[-stealth] (b23)--(b323);

\draw[-stealth] (b223)--(b223123);
\draw[-stealth] (b2123)--(b223123);
\draw[-stealth] (b23123)--(b223123);

\draw[-stealth] (b212)--(b212123);
\draw[-stealth] (b2123)--(b212123);
\draw[-stealth] (b12123)--(b212123);

\draw[-stealth] (b112)--(b112123);
\draw[-stealth] (b1123)--(b112123);
\draw[-stealth] (b12123)--(b112123);

\draw[-stealth] (b113)--(b113123);
\draw[-stealth] (b1123)--(b113123);
\draw[-stealth] (b13123)--(b113123);

\draw[-stealth] (b313)--(b313123);
\draw[-stealth] (b3123)--(b313123);
\draw[-stealth] (b13123)--(b313123);

\draw[-stealth] (b323)--(b323123);
\draw[-stealth] (b3123)--(b323123);
\draw[-stealth] (b23123)--(b323123);
\end{scope}
 
 \end{tikzpicture}
 \end{center}

Note that the $\K_{h}(c \Sd^2\partial d_n\Delta[n])$-part of $\mathcal{P}$ is not this green simplex, but merely the red highlighted part in the following picture:

\begin{center}
 \begin{tikzpicture}[scale=0.8]
 \begin{scope}
 \coordinate (a1) at (0,0); 
 \coordinate (a2) at (7,0);
 \coordinate (a3) at (3.5,6.062);
 
\draw[fill, red](barycentric cs:a1=1,a2=0,a3=0) circle (2pt)node (b1){};
\draw[fill, red](barycentric cs:a1=0,a2=1,a3=0) circle (2pt)node (b2){};
\draw[fill](barycentric cs:a1=0,a2=0,a3=1.0) circle (2pt)node (b3){};
   
\draw[fill, red] (barycentric cs:a1=1,a2=1,a3=1) circle(2pt) node (b123){};

\draw[fill](barycentric cs:a1=1,a2=1) circle (2pt) node (b12){};
\draw[fill](barycentric cs:a2=1,a3=1) circle (2pt) node (b23){};
\draw[fill](barycentric cs:a1=1,a3=1) circle (2pt) node (b13){};

\draw[fill](barycentric cs:b1=1,b12=1) circle (2pt) node (b112){};
\draw[fill](barycentric cs:b2=1,b12=1) circle (2pt) node (b212){};
\draw[fill](barycentric cs:b1=1,b13=1) circle (2pt) node (b113){};

\draw[fill](barycentric cs:b3=1,b13=1) circle (2pt) node (b313){};
\draw[fill](barycentric cs:b2=1,b23=1) circle (2pt) node (b223){};
\draw[fill](barycentric cs:b3=1,b23=1) circle (2pt) node (b323){};


\draw[fill, red](barycentric cs:a1=1,b123=1) circle (2pt) node (b1123){};
\draw[fill, red](barycentric cs:a2=1,b123=1) circle (2pt) node (b2123){};
\draw[fill](barycentric cs:a3=1,b123=1) circle (2pt) node (b3123){};

\draw[fill](barycentric cs:b12=1,b123=1) circle (2pt) node (b12123){};
\draw[fill](barycentric cs:b23=1,b123=1) circle (2pt) node (b23123){};
\draw[fill](barycentric cs:b13=1,b123=1) circle (2pt) node (b13123){};

\draw[fill](barycentric cs:b1=1,b12=1,b123=1) circle (2pt) node (b112123){};
\draw[fill](barycentric cs:b2=1,b23=1,b123=1) circle (2pt) node (b223123){};
\draw[fill](barycentric cs:b3=1,b13=1,b123=1) circle (2pt) node (b313123){};

\draw[fill](barycentric cs:b2=1,b12=1,b123=1) circle (2pt) node (b212123){};
\draw[fill](barycentric cs:b3=1,b23=1,b123=1) circle (2pt) node (b323123){};
\draw[fill](barycentric cs:b1=1,b13=1,b123=1) circle (2pt) node (b113123){};

\draw[-stealth] (b1)--(b112);
\draw[-stealth] (b1)--(b113);
\draw[-stealth] (b2)--(b212);
\draw[-stealth] (b2)--(b223);
\draw[-stealth] (b3)--(b313);
\draw[-stealth] (b3)--(b323);
\draw[-stealth, red] (b1)--(b1123);
\draw[-stealth, red] (b2)--(b2123);
\draw[-stealth] (b3)--(b3123);

\draw[-stealth, red] (b123)--(b1123);
\draw[-stealth, red] (b123)--(b2123);
\draw[-stealth] (b123)--(b3123);
\draw[-stealth] (b123)--(b12123);
\draw[-stealth] (b123)--(b23123);
\draw[-stealth] (b123)--(b13123);

\draw[-stealth] (b12)--(b12123);
\draw[-stealth] (b23)--(b23123);
\draw[-stealth] (b13)--(b13123);
\draw[-stealth] (b12)--(b112);
\draw[-stealth] (b13)--(b113);
\draw[-stealth] (b12)--(b212);
\draw[-stealth] (b23)--(b223);
\draw[-stealth] (b13)--(b313);
\draw[-stealth] (b23)--(b323);

\draw[-stealth] (b223)--(b223123);
\draw[-stealth] (b2123)--(b223123);
\draw[-stealth] (b23123)--(b223123);

\draw[-stealth] (b212)--(b212123);
\draw[-stealth] (b2123)--(b212123);
\draw[-stealth] (b12123)--(b212123);

\draw[-stealth] (b112)--(b112123);
\draw[-stealth] (b1123)--(b112123);
\draw[-stealth] (b12123)--(b112123);

\draw[-stealth] (b113)--(b113123);
\draw[-stealth] (b1123)--(b113123);
\draw[-stealth] (b13123)--(b113123);

\draw[-stealth] (b313)--(b313123);
\draw[-stealth] (b3123)--(b313123);
\draw[-stealth] (b13123)--(b313123);

\draw[-stealth] (b323)--(b323123);
\draw[-stealth] (b3123)--(b323123);
\draw[-stealth] (b23123)--(b323123);
\end{scope}
 
 \end{tikzpicture}
 \end{center}

Then $c\Sd^2\Delta[n]$ can be viewed as glued from $\mathcal{P}$ with $\K_h(c\Sd^2\Delta[n-1])$ along $c\Sd^2\Delta[n-1] \times 0_h$ identified with $\K_h(c\Sd²\partial d_n\Delta[n])$.\\ The $c\Sd^2\Delta[n-1]\times 0_h$ and $c\Sd^2\Delta[n-1]\times 1_h$-parts of this $\K_h(c\Sd^2\Delta[n-1])$ are highlighted red and blue in the picture, respectively; the additional point $k_h$ is the green one. 

\begin{center}
 \begin{tikzpicture}[scale=0.8]
 \begin{scope}
 \coordinate (a1) at (0,0); 
 \coordinate (a2) at (7,0);
 \coordinate (a3) at (3.5,6.062);
 
\draw[fill, red](barycentric cs:a1=1,a2=0,a3=0) circle (2pt)node (b1){};
\draw[fill, red](barycentric cs:a1=0,a2=1,a3=0) circle (2pt)node (b2){};
\draw[fill](barycentric cs:a1=0,a2=0,a3=1.0) circle (2pt)node (b3){};
   
\draw[fill, red] (barycentric cs:a1=1,a2=1,a3=1) circle(2pt) node (b123){};

\draw[fill, green](barycentric cs:a1=1,a2=1) circle (2pt) node (b12){};
\draw[fill](barycentric cs:a2=1,a3=1) circle (2pt) node (b23){};
\draw[fill](barycentric cs:a1=1,a3=1) circle (2pt) node (b13){};

\draw[fill, blue](barycentric cs:b1=1,b12=1) circle (2pt) node (b112){};
\draw[fill, blue](barycentric cs:b2=1,b12=1) circle (2pt) node (b212){};
\draw[fill](barycentric cs:b1=1,b13=1) circle (2pt) node (b113){};

\draw[fill](barycentric cs:b3=1,b13=1) circle (2pt) node (b313){};
\draw[fill](barycentric cs:b2=1,b23=1) circle (2pt) node (b223){};
\draw[fill](barycentric cs:b3=1,b23=1) circle (2pt) node (b323){};


\draw[fill, red](barycentric cs:a1=1,b123=1) circle (2pt) node (b1123){};
\draw[fill, red](barycentric cs:a2=1,b123=1) circle (2pt) node (b2123){};
\draw[fill](barycentric cs:a3=1,b123=1) circle (2pt) node (b3123){};

\draw[fill, blue](barycentric cs:b12=1,b123=1) circle (2pt) node (b12123){};
\draw[fill](barycentric cs:b23=1,b123=1) circle (2pt) node (b23123){};
\draw[fill](barycentric cs:b13=1,b123=1) circle (2pt) node (b13123){};

\draw[fill, blue](barycentric cs:b1=1,b12=1,b123=1) circle (2pt) node (b112123){};
\draw[fill](barycentric cs:b2=1,b23=1,b123=1) circle (2pt) node (b223123){};
\draw[fill](barycentric cs:b3=1,b13=1,b123=1) circle (2pt) node (b313123){};

\draw[fill, blue](barycentric cs:b2=1,b12=1,b123=1) circle (2pt) node (b212123){};
\draw[fill](barycentric cs:b3=1,b23=1,b123=1) circle (2pt) node (b323123){};
\draw[fill](barycentric cs:b1=1,b13=1,b123=1) circle (2pt) node (b113123){};

\draw[-stealth] (b1)--(b112);
\draw[-stealth] (b1)--(b113);
\draw[-stealth] (b2)--(b212);
\draw[-stealth] (b2)--(b223);
\draw[-stealth] (b3)--(b313);
\draw[-stealth] (b3)--(b323);
\draw[-stealth, red] (b1)--(b1123);
\draw[-stealth, red] (b2)--(b2123);
\draw[-stealth] (b3)--(b3123);

\draw[-stealth, red] (b123)--(b1123);
\draw[-stealth, red] (b123)--(b2123);
\draw[-stealth] (b123)--(b3123);
\draw[-stealth] (b123)--(b12123);
\draw[-stealth] (b123)--(b23123);
\draw[-stealth] (b123)--(b13123);

\draw[-stealth] (b12)--(b12123);
\draw[-stealth] (b23)--(b23123);
\draw[-stealth] (b13)--(b13123);
\draw[-stealth] (b12)--(b112);
\draw[-stealth] (b13)--(b113);
\draw[-stealth] (b12)--(b212);
\draw[-stealth] (b23)--(b223);
\draw[-stealth] (b13)--(b313);
\draw[-stealth] (b23)--(b323);

\draw[-stealth] (b223)--(b223123);
\draw[-stealth] (b2123)--(b223123);
\draw[-stealth] (b23123)--(b223123);

\draw[-stealth, blue] (b212)--(b212123);
\draw[-stealth] (b2123)--(b212123);
\draw[-stealth, blue] (b12123)--(b212123);

\draw[-stealth, blue] (b112)--(b112123);
\draw[-stealth] (b1123)--(b112123);
\draw[-stealth, blue] (b12123)--(b112123);

\draw[-stealth] (b113)--(b113123);
\draw[-stealth] (b1123)--(b113123);
\draw[-stealth] (b13123)--(b113123);

\draw[-stealth] (b313)--(b313123);
\draw[-stealth] (b3123)--(b313123);
\draw[-stealth] (b13123)--(b313123);

\draw[-stealth] (b323)--(b323123);
\draw[-stealth] (b3123)--(b323123);
\draw[-stealth] (b23123)--(b323123);
\end{scope}
 
 \end{tikzpicture}
 \end{center}

After Step 1 we already have an extension to $\K_v(\mathcal{P})$, so in particular to $\K_v(\K_h(c\Sd^2\Delta[n-1]))$ (i.e., to $\K_v$ of the red part of the last picture), and we want to promote it to an extension for the whole (doubly subdivided) simplex. A naive attempt would be to build pushouts for all the highlighted spans in the following picture:

\begin{center}
\includegraphics[clip=true, trim = 110 480 80 105]{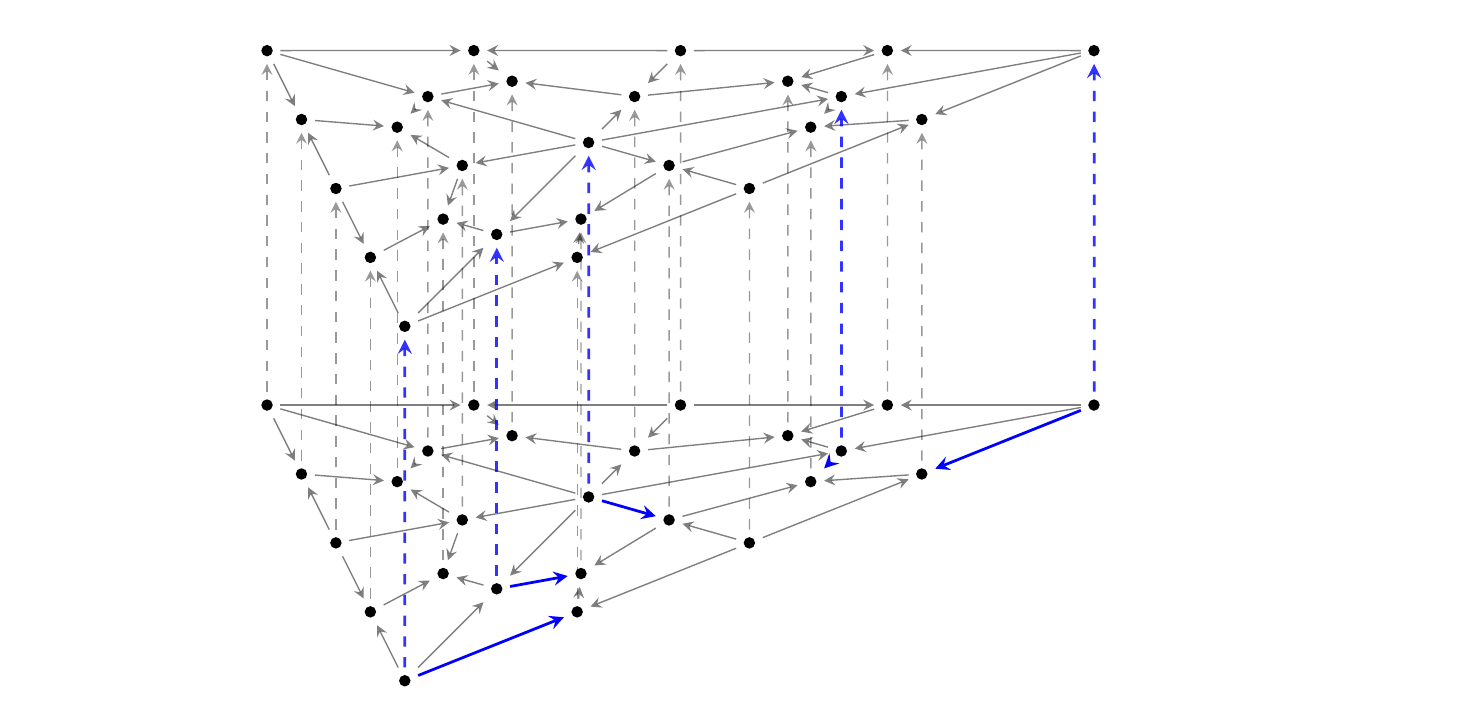}
\end{center}

Yet this does not seem to work without modification. In \ref{Step4}, we take this pushout and modify it by applying $\Psi_{n-1}$ to the pushout result, for n=2 the highlighted part of the picture. 

\begin{center}
\includegraphics[clip=true, trim = 110 480 80 105]{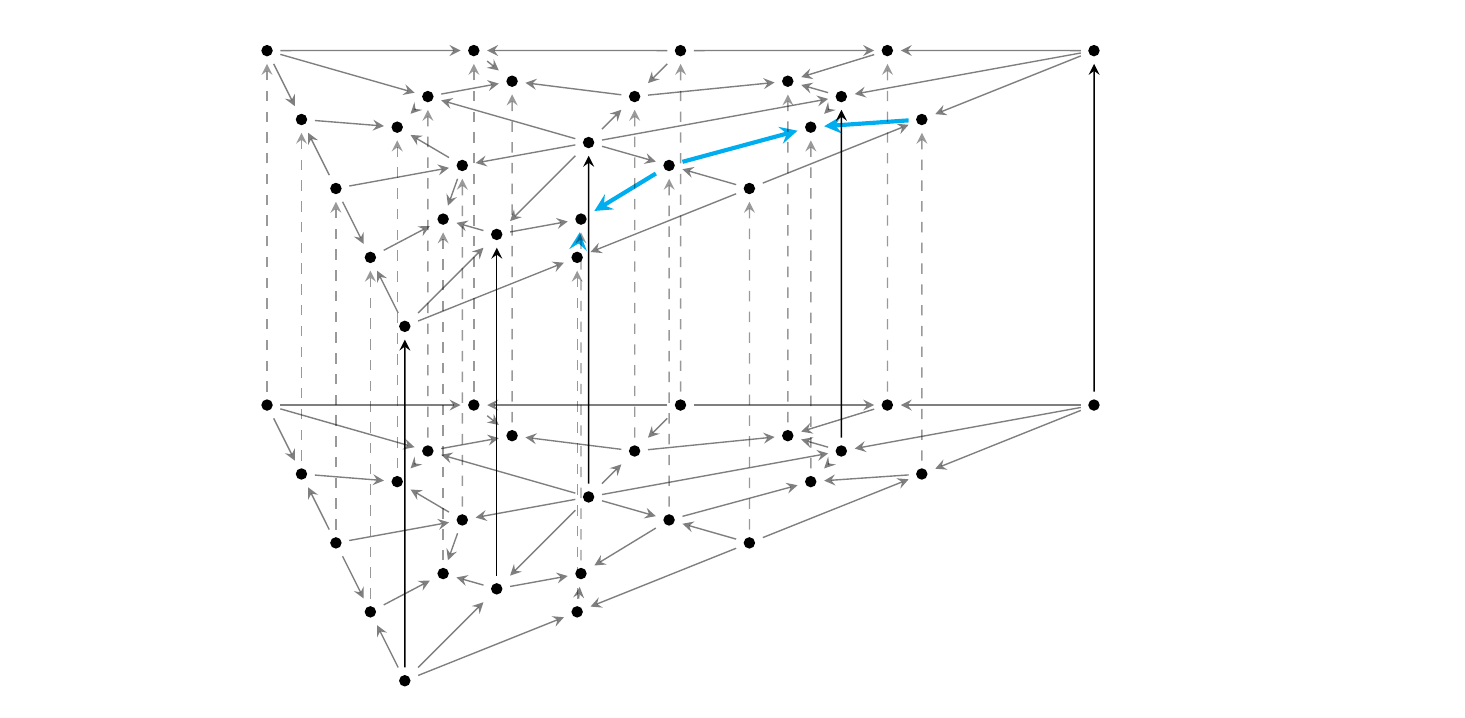}
\end{center}

This modification turns out to be the right one. The extension so far is shown to satisfy \ref{F1} in \ref{Step4} and \ref{F3} in \ref{Step5}. In \ref{Step6}, we extend our functor to the last missing part (corresponding to the $k_h$ of $\K_h(c\Sd^2\Delta[n-1])$), and we check it still has the correct limit property in \ref{Step7}. This completes the construction of $\Psi_n$. We construct $\Phi_{n+1}$ out of $\Psi_n$ and $\Phi_n$ using the boundary conditions in \ref{Step8}. Here, we exploit the fact that an $(n+1)$-horn is obtained by gluing $n$-simplices along $(n-1)$-simplices.

 Now we begin with the actual proof. First, note that $\Phi_1$ can be just chosen to extend a given functor by a constant functor since $c\Sd^2\Lambda^1[1]$ is the category $\underline{0}$ with 
 exactly one object and no non-trivial morphisms. In the same manner, we construct $\Psi_0$ and observe that the condition \ref{F4} is satisfied for $n=0$. We proceed by induction. Assume that extension functors $\Phi_{n}$ and $\Psi_{n-1}$ (and all smaller ones) with properties \ref{F1}, \ref{F3} and \ref{F4}--\ref{F5} have already been constructed, and let $\alpha\colon c\Sd^2 \Delta[n] \to \mathcal{W}$ be given. We would like to extend $\alpha$ to a functor $\K(c\Sd^2 \Delta[n]) \to \mathcal{W}$, yielding the functor $\Psi_n$. The reader should observe that all the steps of our construction are functorial, so $\Phi$ and $\Psi$ are going to be functors. \vspace*{0.3cm}

 \begin{enumerate}[label=\textbf{Step \arabic{enumi}:}, ref=Step \arabic{enumi}]
  \item\label{Step1}  To extend $\alpha$ to $\K_v\K_h(c\Sd^2\Lambda^n[n])$, we apply Lemma \ref{DoubleK} with $\D=c\Sd^2\Lambda^n[n]$. By induction hypothesis, the functor $\Phi_n$ is an extension functor satisfying the hypothesis of Lemma \ref{DoubleK}. This lemma gives us (functorially) a functor $\alpha' \colon \K_v(\mathcal{P})= \K_v\K_h(c\Sd^2\Lambda^n[n])\to \mathcal{W}$ extending $\alpha$. In addition, it satisfies \ref{F1} and \ref{F3} and coincides with $\Phi_n(\alpha|_{c\Sd^2\Lambda^n[n]})$ when restricted to $c\Sd^2\Lambda^n[n]\times (0\to 1)_v$. The last part of Lemma \ref{DoubleK} gives us a compatible fibration 
	\[\alpha'(k_v) \to \Phi_n(\alpha|_{c\Sd^2\Lambda^n[n]})(k_v).\] 
	In particular, the map 
  \begin{eqnarray*}
   \alpha'(k_v) \to \alpha'(\{n\}, 1_v)
  \end{eqnarray*}
 is a fibration as a composition of two fibrations (using Property \ref{F5} for $\Phi_n$). The reader should also observe that during the next steps in the construction of $\Psi_n$, the value on $c\Sd^2\Lambda^n[n] \times 0_h\times (0\to 1)_v$ and thus in particular on $(\{n\}, 1_v)$ will not be changed. \vspace*{0.3cm}

\item\label{Step4} We want to extend $\alpha'$ now somewhat further.  Recall that we have already defined a functor $\alpha'\colon \K_v(\mathcal{P}) \to \mathcal{W}$ which maps morphisms of the form $x \times (0\to 1)_v$ to cofibrations and where $\alpha'(k_v)$ is the limit of $\alpha'|_{\mathcal{P}\times 1_v}$. Let $\mathcal{P}'$ stand short for
\[\mathcal{P}\cup_{c\Sd^2\Delta[n-1]} (c\Sd^2\Delta[n-1]\times (0\to 1)_h), \] 
i.e., $c\Sd^2 \Delta[n]$ without the barycenter of $d_n\Delta[n]$. Here, $c\Sd^2 \Delta[n-1]$ is still identified with $\K_h(\partial d_n\Delta[n])$ as in the description of the plan of the proof. Now we would like to extend both $\alpha'$ and $\alpha$ simultaneously to $\KK_v(\mathcal{P}')$. Observe that for each object $x$ of $c\Sd^2\Delta[n-1]$, we already have a span of morphisms

\vspace{0.4cm}
\hbox{
\begin{tikzpicture}[node distance=1.6cm, auto,  >=stealth, scale=0.8]
  \node (P) {$\alpha(x \times 0_h\times 0_v)=\alpha'(x \times 0_h\times 0_v)$};
  \node (B) [below of=P] {$\alpha'(x \times 0_h\times 1_v)$};
  \node (A) [right of=P,xshift = 6.5cm] {$\alpha(x\times 1_h\times 0_v)$};
  \draw[->] (P) to node{$\alpha'(x\times 0_h \times (0\to 1)_v)$} (B);
  \draw[->] (P) to node{$\alpha(x\times (0\to 1)_h \times 0_v)$} (A);
\end{tikzpicture}\vspace{0.4cm}
}

\noindent and the horizontal map is a cofibration by construction. Thus we could build all the pushouts of such squares to obtain values of our extension $\wt{\alpha}$ on all of $\K_v(\mathcal{P}')$,
 since we can simply take compositions to obtain missing maps from $\alpha'(k_v)$. Yet this doesn't seem to work, and we have to modify this attempt somewhat further. Taking pushouts and induced maps yields a functor $\gamma \colon c\Sd^2\Delta[n-1] \to \mathcal{W}$, where $c\Sd^2\Delta[n-1]$ is identified with its $1_h\times 1_v$-copy. We can apply $\Psi_{n-1}$ to it. We define our new extension $\wt{\alpha}$ to coincide with $\alpha'$ and $\alpha$ wherever it makes sense, and let it be $\Psi_{n-1}(\gamma)|_{c\Sd^2\Delta[n-1]\times 1}$ on $c\Sd^2\Delta[n-1] \times 1_h\times 1_v$. Since there are no morphisms from these to objects not of this form in $\K_v(\mathcal{P}')$, we have to take care only 
of morphisms with targets in $c\Sd^2\Delta[n-1] \times 1_h\times 1_v$. To do so, we compose the pushout morphisms we obtained for $\gamma$ with the corresponding maps in $\Psi_{n-1}(\gamma)|_{c\Sd^2\Delta[n-1] \times (0\to 1)}$ as depicted in the following diagram for an $a\in c\Sd^2\Delta[n-1]$:
\[\xymatrix{
\wt{\alpha}(k_v) \ar@{-->}[drr] \ar[ddr]\\
&&\wt{\alpha}(a\times 1_h\times 1_v) = \Psi_n(\gamma)(a\times 1_h\times 1_v\times 1) \\
&\wt{\alpha}(a\times 0_h\times 1_v)\ar@{-->}[ur] \ar[r] & \gamma(a\times 1_h \times 1_v)\ar@{ >->}[u] \\
&\wt{\alpha}(a\times 0_h\times 0_v) \ar[r]\ar@{ >->}[u]\ar@{}[ur]|{\text{\pigpenfont L}} & \wt{\alpha}(a\times 1_h\times 0_v) \ar@{-->}@/_3.6pc/[uu]\ar@{ >->}[u]
}
\]
 This yields a functor 
\[
 \wt{\alpha}\colon \K_v(\mathcal{P}') \to \mathcal{W}.
\]
Moreover, note that all pushout morphisms of the form 
\[\alpha(x\times 1_h\times 0_v) \to \gamma(x) \]
 are cofibrations. By Property \ref{F1}, the maps these are composed with are also cofibrations, so that for any $x \in c\Sd^2\Delta[n-1] \times 1_h$, the morphism $\wt{\alpha}(x \times (0\to 1)_v)$ is a cofibration.\vspace*{0.3cm}

\item \label{Step5} We want to show that $\wt{\alpha}$ satisfies again a limit property analogous to \ref{F3}, namely that the object $\wt{\alpha}(k_v)$ is a limit of the diagram $\wt{\alpha}|_{\mathcal{P}' \times 1_v}$. It is enough to show that the inclusion functor $\iota \colon\mathcal{P} \to \mathcal{P}'$ is an initial functor since we already know that $\wt{\alpha}(k_v)$ is a limit of $\wt{\alpha}|_{\mathcal{P}\times 1_h}$. 

For doing so, we have to show that for any object $x$ of $\mathcal{P}'$, the comma category 
\[(\iota\downarrow x)\]
 is non-empty and connected. This is clear for objects $x$ in $\mathcal{P}$ since in this case $(\iota\downarrow x)$ has the terminal object $\id_{x}$. Now consider an object $(y, 1_h)$  where $y$ is an object of $c\Sd^2\Delta[n-1]$. Then $(y,0_h) \to (y, 1_h)$ is the terminal object of the category $(\iota\downarrow (y,1_h))$. Thus, the functor $\iota$ is initial and the object $\wt{\alpha}(k_v)$ is a limit of the diagram $\wt{\alpha}|_{\mathcal{P}' \times 1_v}$. 

Furthermore, observe that since $\alpha'(k_v)=\wt{\alpha}(k_v)$ and the maps to $\wt{\alpha}(\{n\},1_v)$ coincide, the map $\wt{\alpha}(k_v) \to \wt{\alpha}(\{n\}, 1_v)$ is a fibration.\vspace*{0.3cm}

\item \label{Step6} The functor $\wt{\alpha}$ is undefined at $k_h\times 1_v$ for $k_h$ in $\K_h(c\Sd^2\Delta[n-1])$ the barycenter of $d_n\Delta[n]$. We want to extend $\wt{\alpha}$ to a functor 
\[ \delta\colon \K(c\Sd^2\Delta[n]) \to \mathcal{W}.\]
 On objects of $\K_v(\mathcal{P}')$ except $k_v$, define $\delta$ to coincide with $\wt{\alpha}$. Our first try would be to define $\delta(k_h\times 1_v)$ to be the value at the additional point $\Psi_{n-1}(\gamma)(k)$. This is again almost what we do. This object already has all necessary maps: by definition, we have maps to $\delta(c\Sd^2\Delta[n-1] \times 1_h\times 1_v)$. Since by Property \ref{F3}, the object $\Psi_{n-1}(\gamma)(k)$ is a limit of $\delta(c\Sd^2\Delta[n-1] \times 1_h\times 1_v)$ and since we have compatible maps from $\delta(k_h\times 0_v)$ to $\delta(c\Sd^2\Delta[n-1] \times 1_h\times 1_v)$, we get also a map $\delta(k_h\times 0_v)$ to the limit $\Psi_{n-1}(\gamma)(k)$. For a similar reason, we obtain a map $\wt{\alpha}(k_v)\to \Psi_{n-1}(\gamma)(k)$. As the map $\delta(k_h\times 0_v)\to\Psi_{n-1}(\gamma)(k)$ is not necessarily a cofibration, we first factorize this map into a cofibration followed by a fibration using functorial factorization: 
\begin{eqnarray*}
 \delta(k_h\times 0_v)\xrightarrow{g'} D \xrightarrow{h'} \Psi_{n-1}(\gamma)(k).
\end{eqnarray*}
We define $\delta(k_h \times 1_v)=D$, and $g'$ to be the image of $k_h\times (0\to 1)_v$ under $\delta$. Analogously to \ref{LStep2} of Lemma \ref{DoubleK}, we define $\delta(k_v)$ to be the following pullback:
\[\xymatrix{\delta(k_v)\ar[r]\ar@{->>}[d]&D \ar@{->>}[d]^{h'} \\ \wt{\alpha}(k_v) \ar[r]\ar@{}[ur]|{\text{\pigpenfont J}} & \Psi_{n-1}(\gamma)(k) }\] 
 This is possible since $h'$ is defined to be a fibration. As before, it also implies that $\delta(k_v)\to \wt{\alpha}(k_v)$ is a fibration. Again, besides the two morphisms already defined, all other morphisms in $\K(c\Sd^2\Delta[n])$ involving $k_v$ start in $k_v$, so the corresponding image morphism under $\delta$ is obtained by precomposing with $h'$ the corresponding map from $\Psi_{n-1}(\gamma)(k)$. This defines a functor 
\begin{eqnarray*}
 \delta\colon \K(c\Sd^2\Delta[n]) \to \mathcal{W}.
\end{eqnarray*}
We define $\Psi_n(\alpha)$ to be $\delta$. First, observe that this construction is functorial in $\alpha$, since we only used constructions like functorial factorization, limits and colimits and previously constructed functors $\Psi_{n-1}$ and $\Phi_n$ to define $\Psi_n$. Moreover, this is by construction an extension functor satisfying condition \ref{F1}. Also, \ref{F5} is satisfied as the map $\delta(k_v) \to \delta(\{n\}, 1_v)$ is a composition of two fibrations.\vspace*{0.3cm}

\item \label{Step7} Next, we want to prove that \ref{F3} is satisfied for our newly defined functor $\Psi_n$. Note that in the diagram $c\Sd^2\Delta[n]\times 1_v$, the only object not in $\mathcal{P}'\times 1_v$ is $k_h\times 1_v$ and, as mentioned above, it has only outgoing maps to objects of $\mathcal{P}'\times 1_v$, more precisely, to objects of $c\Sd^2\Delta[n-1]\times 1_h\times 1_v$ inside it. Now assume that we have an object $T$ in $\mathcal{W}$ with compatible morphisms to $\delta(c\Sd^2\Delta[n]\times 1_v)$. In particular, we can consider only the subset of these morphisms with targets in $\delta(\mathcal{P}')$, inducing a morphism $T\to \wt{\alpha}(k_v)$. Considering the even smaller subset of maps to $\delta(c\Sd^2\Delta[n-1]\times 1_h\times 1_v)$, we obtain a map to $\Psi_{n-1}(\gamma)(k)$ which is compatible with the previous one. We have not yet
considered the map $T\to \delta(k_h\times 1_v)=D$. Observe that the composition $T\to D \to \Psi_{n-1}(\gamma)(k)$ is the same map as the already described one. Thus, by the universal property of the pullback, we obtain a compatible map $T\to \delta(k_v)$. Similarly, it is not hard to see the uniqueness. Thus, $\delta(k_v)$ is indeed a limit of the restriction of $\delta$ to $c\Sd^2\Delta[n]\times 1_v$, and the functor $\Psi_n$ satisfies \ref{F3}. In particular, $\Psi_n$ has now all desired properties.\vspace*{0.3cm}

\item \label{Step8} We continue by constructing the functor $\Phi_{n+1}$. Let 
\[\alpha\colon c\Sd^2\Lambda^n[n]\to \mathcal{W}\]
 be given. Property \ref{F4} dictates us how to extend $\alpha$ to $c\Sd²\Lambda^n[n] \times (0\to 1)$. This is possible by Property \ref{F6}. In particular, it implies that $\Psi_n(\alpha|_{d_i\Delta[n+1]})(\{n+1\},1_v)$ is the same for all $0\leq i\leq n$; we call this object $E$. Now we need to define $\Phi_{n+1}(\alpha)(k)$. 
We have fibrations $\Psi_n(\alpha|_{d_i\Delta[n+1]})(k_v)\to E$ for all $0\leq i\leq n$. Successively building pullbacks, we obtain a new object $X$ in $\mathcal{W}$, together with fibrations $X \to \Psi_n(\alpha|_{d_i\Delta[n+1]})(k_v)$ for all $0\leq i\leq n$, which give us the same fibration $X\to E$ when continued to $E$. Moreover, $X$ is the limit of the diagram given by maps $\Psi_n(\alpha|_{d_i\Delta[n+1]})(k_v)\to E$. Define $\Phi_{n+1}(\alpha)(k)=X$. By construction, the Properties \ref{F1} and \ref{F4} as well as \ref{F5} are obviously satisfied. So we are left with proving \ref{F3}. This is completely analogous to the argument in \ref{Step7}.

 \end{enumerate}

All in all, this completes the induction step and thus the proof of the theorem. 
 
\end{pf}

\section{Examples}
\label{Examples}
In this section, we would like to discuss some examples. Obviously, weak equivalences $\WW$ of any model category satisfy the conditions of Theorem \ref{MainTheorem}, thus providing a wide range of examples. Here, we would like to present some of these and a further example of partial model category. In particular, we want to show examples of partial model categories $(\C, \WW)$ where $\Ex^2 N\WW$ is Kan, but $\Ex N\WW$ is not Kan. This includes a category with all pullbacks. In contrast, we will show that for a category $\CC$ with all pushouts, $\Ex N\CC$ is already Kan. Furthermore, we will give an example of a category $\CC$ with $\Ex N\CC$ Kan, which is not the category of weak equivalences of a partial model category. This shows that the hypotheses of Theorem \ref{MainTheorem} are sufficient, but not necessary. At last, we note that for every filtered category $\CC$ the simplicial set $\Ex N\CC$ is Kan.
\\

Recall that any category with all pushouts or all pullbacks can be given a structure of a partial model category with every morphism being a weak equivalence. 

\begin{example}\label{FIExample}
 It is easy to show that the category $FI$ of finite sets and injections has all pullbacks, so $\Ex^2N(FI)$ is Kan. But $FI$ does not have the lifting property with respect to the diagram
 \begin{eqnarray*}
  \bullet \to \bullet\rightrightarrows \bullet \dashrightarrow \bullet
 \end{eqnarray*}
For example, we can consider 
 \[\xymatrix{
\{0\} \ar[r] & \{0,1,2\}\ar@/^/[r]^{\id} \ar@/_/[r]_{\tau} & \{0,1,2\},
}\]
where $\tau$ permutes $1$ and $2$. Thus, $\Ex N(FI)$ is not Kan. 
 
It is furthermore interesting to note that while property \ref{CF1} holds for $FI$, the pushout in the category of sets of a given diagram of finite sets and injections between them is not necessarily a pushout in the category $FI$. For example, consider
\[\xymatrix{\{1\} \ar[r] \ar[d]& \{1,2\}\\
\{0,1\} 
}\]
and the maps maps $\id\colon \{1,2\} \to \{1,2\}$ and $q\colon \{0,1\}\to \{1,2\}$ with $q(0) = 2$ and $q(1) = 1$. Then the induced map $\{0,1,2\}\to \{1,2\}$ is not an injection. It is easy to deduce that this diagram actually does not have a pushout in $FI$. 
\end{example}

In contrast to this example, every category $\WW$ with all pushouts already possesses left calculus of fractions with respect to itself. Indeed, even more is true; we will prove the following easy proposition. This is essentially contained in \cite{Pare}, but we prefer to give the following direct proof. 

\begin{Prop}\label{PushoutProp}
 Let $\WW$ be a category with all pushouts. Let the diagram
 \[\xymatrix{
A \ar[r]^{\alpha} & X \ar@/^/[r]^{f} \ar@/_/[r]_{g} & Y
}\]
with $f\alpha=g\alpha$ be given. Then the diagram
 \[\xymatrix{
 X \ar@/^/[r]^{f} \ar@/_/[r]_{g} & Y
}\]
has a coequalizer. In particular, both properties \ref{CF1} and \ref{CF2} hold in $\WW$. Thus $\WW$ possesses left calculus of fractions with respect to itself. 
\end{Prop}

\begin{pf}
 First, we build the pushout $B$ as follows:
 \begin{eqnarray*}
\xymatrix{ 
A \ar[r]^{\alpha}\ar[d]_{f\alpha=g\alpha} & X\ar[d]^{i}\\
Y\ar[r]_{j} & B. 
}  
 \end{eqnarray*}
 From $B$, we obtain two maps $F, G\colon B \to Y$ from
 \begin{eqnarray*}
  \xymatrix{
  A\ar[r]^{\alpha} \ar[d]_{f\alpha} & X\ar[d]^{i}\ar@/^/[rdd]^{f} &\\
  Y \ar[r]_{j}\ar@/_/[rrd]_{\id}& B\ar@{-->}[rd]_{F}& \\
  & & Y
  }
 \end{eqnarray*}
 and 
 \begin{eqnarray*}
  \xymatrix{
  A\ar[r]^{\alpha} \ar[d]_{g\alpha} & X\ar[d]^{i}\ar@/^/[rdd]^{g} &\\
  Y \ar[r]_{j}\ar@/_/[rrd]_{\id}& B\ar@{-->}[rd]_{G}& \\
  & & Y.
  }
 \end{eqnarray*}
 
We claim that the pushout $Z$ defined in the diagram below is the desired coequalizer:
\begin{eqnarray*}
 \xymatrix{ 
B \ar[r]^{G}\ar[d]_{F} & Y\ar[d]^{\varphi_1}\\
Y\ar[r]_{\varphi_2} & Z. 
}  
\end{eqnarray*}

First, we want to show that $\varphi_1=\varphi_2$. This holds since we can precompose the equation $\varphi_1 G=\varphi_2 F$ with $j\colon Y \to B$ and obtain by definition of $F$ and $G$: 
\begin{eqnarray*}
 \varphi_1=\varphi_1 \circ \id_Y=\varphi_1 G j= \varphi_2 F j=\varphi_2.
\end{eqnarray*}

Next, we claim that $\varphi=\varphi_1=\varphi_2\colon Y \to Z$ equalizes $f$ and $g$. Indeed, precompose the identity $\varphi_1 G=\varphi_2 F$ with $i\colon X \to B$ and obtain again by definition of $F$ and $G$: 
\begin{eqnarray*}
 \varphi g= \varphi_1 G i  =\varphi_2 F i= \varphi f.
\end{eqnarray*}

Last, we have to show that $\varphi$ is universal with this property. Assume $h\colon Y \to T$ is a map with $hf=hg$. We want to show that $h$ factors uniquely through $Z$, i.e., that there is a unique map $\psi\colon Z \to T$ so that $\psi\circ \varphi=h$. To construct such a map, we need to show that $hG=hF\colon B\to Y$; this will induce a unique map $\psi$ as desired. As $B$ was defined as a pushout, to check this property, it is equivalent to check that $hGj=hFj$ and $hGi=hFi$. As $Fj=Gj=\id_Y$ and $hGi=hg=hf=hFi$, we are done. 
\end{pf}

\begin{example}
Let $\mathrm{we} \Cat$ be the category with objects all (small) categories and as morphisms functors that induce weak equivalences on nerves. As these are the weak equivalences in the Thomason model structure, this falls inside the scope of our theorem. We want to show that Property \ref{CF1} fails for $\mathrm{we}\Cat$. 

Let $\CC$ be the category with two objects $x$ and $y$ and two non-identity morphisms $a$ and $b$ from $x$ to $y$, depicted as follows:
 \[\xymatrix{
 x \ar@/^/[r]^{a} \ar@/_/[r]_{b} & y
}\]
Let furthermore $\mathbb{N}$ denote the category with one object corresponding to the monoid of natural numbers (with addition) and call the generating morphism $t$.
Then we consider the diagram
\[\xymatrix{ \CC \ar[r]^F \ar[d]_G & \mathbb{N} \\
\mathbb{N} 
}\]
  The functor $F$ sends $a$ to $t$ and $b$ to the identity and the functor $G$ sends $a$ to the identity and $b$ to $t$. We need to show that $NF$ and $NG$ are weak equivalences. As $N\mathbb{N} \to N\mathbb{Z}$ is a weak equivalence (e.g. by \cite[Proposition 4.4]{Fiedorowicz}), we only need to show that the composition $\CC \to \mathbb{Z}$ induces a weak equivalence on nerves. As both nerves are $K(\Z,1)$, we have only to show that the map induces an isomorphism on $\pi_1$. The fundamental group of $\CC$ is generated by $b^{-1}a$. This is mapped by $F$ and $G$ to $1$ and $-1$ in $\mathbb{Z}$, respectively.
 Thus, $NF$ and $NG$ are weak equivalences. 

Now assume we could complete the above diagram to a commutative square
\[\xymatrix{ \CC \ar[r]^F \ar[d]_G & \mathbb{N} \ar[d] \\
\mathbb{N} \ar[r] & \DD
}\]
The resulting functor $\CC \to \DD$ has to send $x$ and $y$ to the same object and $a$ and $b$ to the identity, i.e., it factors over the terminal category $\ast$. As $N\CC$ is not contractible, the functor $\CC \to \DD$ cannot be a weak equivalence. 

Thus, $\Ex N \mathrm{we} \Cat$ is not Kan.
\end{example}

\begin{example}
Virtually the same example works also for the category of weak equivalences $\mathrm{we}\Top$ of $\Top$, either meaning homotopy equivalences or weak homotopy equivalences. Consider the diagram
\[\xymatrix{ S^1 \ar[r]^f\ar[d]_g & S^1\\
S^1 }\]
where $f$ collapses to the lower closed hemisphere to a point and $g$ the upper one. These are clearly homotopy equivalences. Assume we could complete the diagram to a commutative square
\[\xymatrix{ S^1 \ar[r]^f\ar[d]_g & S^1\ar[d]\\
S^1 \ar[r] & X }\]
Then the resulting map $S^1 \to X$ must send all points of the lower hemisphere and all points of the upper hemisphere to the same point. Hence, it factors over a point and cannot be a (weak) homotopy equivalence. 

Thus, $\Ex N \mathrm{we}\Top$ is not Kan.
\end{example}

\begin{example}\label{MonoidExample}
Consider the monoid $M$ generated by one element $a$ with the relation $a^2 = a$. We view it as a category with one object and one non-trivial morphism $a$. It has the property that $ax = a = ay$ for any morphisms $x,y$ in $M$. Thus, $M$ has a left calculus of fraction with respect to itself and $\Ex NM$ is fibrant. 

On the other hand, $M$ is not the category of weak equivalences of a partial model category. Assume otherwise. If the class of acyclic cofibrations contains $a$, then the pushout of the diagram
\[\xymatrix{\bullet \ar[r]^a\ar[d]_a &  \bullet \\
\bullet }\]
has to exist. Assume that 
\[\xymatrix{\bullet \ar[r]^a\ar[d]_a &  \bullet\ar[d]^b \\
\bullet\ar[r]_c & \bullet }\]
is a pushout. This is equivalent to the following: For every $x,y \in M$ with $xa = ya$, there is a \textit{unique} $z\in M$ with $zba = xa$. As every product involving $a$ equals $a$, this $z$ can always be chosen to be either $1$ or $a$, so it cannot be unique and the pushout of the diagram above does not exist. Thus, the class of acyclic cofibrations has to be equal to $\{\id\}$. As $M$ is self-dual, the same is true for the class of acyclic fibrations. Thus, we cannot factor the morphism $a$ into an acyclic cofibration and an acyclic fibration. 

Thus, $M$ is not the category of weak equivalences of a partial model category. 
\end{example}

\begin{example}
Every filtered category has a left calculus of fraction. Indeed, \ref{CF2} is clear by definition and given a span 
\[\xymatrix{X\ar[r]^s \ar[d]_t& Y \\
Z }\]
we can argue as follows: By the definition of a filtered category, there is an object $W$ with morphisms $f\colon Y\to W$ and $g\colon Z \to W$. Furthermore, there is a morphism $h\colon W\to Q$ with $hfs = hgt$. This defines a commutative square
\[\xymatrix{X\ar[r]^s \ar[d]_t& Y \ar[d]^{hf} \\
Z \ar[r]_{hg} & Q}\]
This implies \ref{CF1}.
\end{example}

\section{Open questions and outlook}
\label{OpenQuestions}
\begin{Question}Is every category equivalent in the Thomason model structure to the category of weak equivalences of a partial model category?
\end{Question}

Note that Example \ref{MonoidExample} shows that not every category that is fibrant in the Thomason model structure is the category of weak equivalences of a partial model category. 

\begin{Question}Is there a notion of a category $\CC$ with an $(n+1)$-arrow calculus (with respect to itself) such that $\Ex^nN\CC$ is fibrant but not necessarily $\Ex^{n-1}N\CC$?
\end{Question}
This question is a strengthening of a conjecture of Beke as formulated in \cite[Conjecture 3.1]{BekeFibrations} that there are for all $n$ categories $\CC$ such that $\Ex^n N\CC$ is fibrant but $\Ex^{n-1}N\CC$ is not. As a main difficulty in this conjecture is the lack of good criteria for $\Ex^n N\CC$ being fibrant, our question seems to be a possible line of attack towards Beke's conjecture. 

Note that being a groupoid is the same as a $1$-arrow calculus and a left calculus of fractions is a form of $2$-arrow calculus. Furthermore, partial model categories are a form of categories with a $3$-arrow calculus. In the best of all worlds, one could hope that there is a (transparent) notion of an $(n+1)$-arrow calculus such that $\Ex^nN\CC$ is fibrant if and only if $\CC$ has an $(n+1)$-arrow calculus. 

\begin{Question}Is every partial model category fibrant in the Barwick--Kan model structure on relative categories as defined in \cite{BKRelative}?\end{Question}
This does not seem to be true, but the question whether model categories are fibrant in the Barwick--Kan model structure will be addressed in a forthcoming paper by the first-named author. 

\begin{Question}Is there an example of a category $\CC$ such that $\Ex^2 N\CC$ is fibrant, but $\Ex^2N\CC^{op}$ is not?\end{Question}
There are a few observations motivating this question: Whether $N\CC$ is Kan is a self-dual property ($\CC$ has it if and only if $\CC^{op}$ has it). Whether $\Ex N\CC$ is Kan is \textit{not} a self-dual property as Example \ref{FIExample} and Proposition \ref{PushoutProp} show. In contrast, whether $\CC$ is the category of weak equivalences of a partial model category is again a self-dual property. 

 \bibliographystyle{plain}
\bibliography{../Garside}
\end{document}